\documentclass[preprint,12pt]{elsarticle}

\RequirePackage{amsthm,amsmath,amsfonts,amssymb,ulem}
\RequirePackage[numbers]{natbib}
\RequirePackage[colorlinks,citecolor=blue,urlcolor=blue]{hyperref}
\RequirePackage{graphicx}
\usepackage{enumerate}

\numberwithin{equation}{section}
\theoremstyle{plain}
\newtheorem{theorem}{Theorem}[section]
\newtheorem{corollary}[theorem]{Corollary}
\newtheorem{lemma}[theorem]{Lemma}
\newtheorem{proposition}[theorem]{Proposition}
\theoremstyle{remark}
\newtheorem{remark}{Remark}[section]
\newtheorem{definition}[remark]{Definition}
\newtheorem{example}[remark]{Example}

\usepackage{tikz}

\newcommand{\col}{{\rm col}}

\newcommand{\sign}{{\rm sign}}
\newcommand{\model}{{\rm patt}}

\newcommand{\aff}{{\rm aff}}

\newcommand{\betaOLS}{\hat\beta^{\rm OLS}}

\newcommand{\betaLASSO}{\hat\beta^{\rm LASSO}}

\newcommand{\betaSLOPE}{\hat\beta^{\rm SLOPE}}

\renewcommand{\P}{\mathbb{P}}

\newcommand{\R}{\mathbb{R}}

\newcommand{\E}{\mathbb{E}}

\DeclareMathOperator*{\argmin}{\arg\min}
\newcommand{\X}{{X}}
\newcommand{\Y}{{Y}}

\newcommand{\Zb}{{Z}}
\newcommand{\U}{{U}}
\newcommand{\I}{{I}}
\renewcommand{\P}{{P}}

\newcommand{\bLambda}{{\Lambda}}

\newcommand{\beps}{\varepsilon}
\newcommand{\Snew}[3]{S_{#1, {#2}}(#3)}
\newcommand{\C}{{C}}

\begin{document}

\begin{frontmatter}


\tnotetext[label1]{T. Skalski was supported by a French
	Government Scholarship.
	Research of \mbox{B. Ko{\l}odziejek} was funded by (POB Cybersecurity and Data Science) of Warsaw University of Technology within the Excellence Initiative: Research University (IDUB) programme. Research of P. Graczyk and T. Skalski was supported by Centre Henri Lebesgue, program ANR-11-LABX-0020-0. Research of P. Tardivel is supported by the region Bourgogne-Franche-Comté (EPADM project). The institute of X. Dupuis and P. Tardivel receives support from the EIPHI Graduate School (contract ANR-17-EURE-0002).}
\cortext[cor1]{Corresponding author.}
\fntext[cor7]{The order of authors is alphabetical.}
\title{Pattern recovery by SLOPE}
\author[inst1,inst2]{Ma\l{}gorzata Bogdan}
\ead{malgorzata.bogdan@uwr.edu.pl}
\author[inst3]{Xavier Dupuis}
\ead{xavier.dupuis@u-bourgogne.fr}
\author[inst4]{Piotr Graczyk}
\ead{piotr.graczyk@univ-angers.fr}
\author[inst5]{Bartosz Ko\l{}odziejek}
\ead{bartosz.kolodziejek@pw.edu.pl}
\author[inst4,inst6]{Tomasz Skalski \corref{cor1}}
\ead{tomasz.skalski@pwr.edu.pl}
\author[inst3]{Patrick Tardivel}
\ead{patrick.tardivel@u-bourgogne.fr}
\author[inst6]{Maciej Wilczy\'nski}
\ead{maciej.wilczynski@pwr.edu.pl}

\affiliation[inst1]{organization={Institute of Mathematics,
University of Wroc\l{}aw},
            addressline={pl. Grunwaldzki 2/4}, 
            city={Wroc\l{}aw},
            postcode={50-384}, 
            country={Poland}}

\affiliation[inst2]{organization={Department of Statistics, Lund University},
            addressline={Holger Crafoords Ekonomicentrum 1, Tycho Brahes väg 1}, 
            city={Lund},
            postcode={SE-220 07}, 
            country={Sweden}}
            
\affiliation[inst3]{organization={Université Bourgogne Europe, CNRS, IMB UMR 5584},
            city={21000 Dijon},
            country={France}}

\affiliation[inst4]{organization={Laboratoire de Mathématiques LAREMA},
            addressline={Université d’Angers, 2 Boulevard Lavoisier}, 
            city={Angers},
            postcode={49045}, 
            country={France}}

\affiliation[inst5]{organization={Faculty of Mathematics and Information Science,
Warsaw University of Technology},
            addressline={Koszykowa 75}, 
            city={Warsaw},
            postcode={00-662}, 
            country={Poland}}

\affiliation[inst6]{organization={Faculty of Pure and Applied Mathematics},
            addressline={Wroc\l{}aw University of Science and Technology, Wybrzeże Wyspiańskiego 27}, 
            city={Wroc\l{}aw},
            postcode={50-370}, 
            country={Poland}}

\begin{abstract}
SLOPE is a popular method for dimensionality reduction in high-dimensional regression. Its estimated coefficients can be zero, yielding sparsity, or equal in absolute value, yielding clustering. As a result, SLOPE can eliminate irrelevant predictors and identify groups of predictors that have the same influence on the response. The concept of the SLOPE pattern allows us to formalize and study its sparsity and clustering properties. In particular, the SLOPE pattern of a coefficient vector captures the signs of its components (positive, negative, or zero), the clusters (groups of coefficients with the same absolute value), and the ranking of those clusters.\\
This is the first paper to thoroughly investigate the consistency of the SLOPE pattern. We establish necessary and sufficient conditions for SLOPE pattern recovery, which in turn enable the derivation of an irrepresentability condition for SLOPE given a fixed design matrix $X$. These results lay the groundwork for a comprehensive asymptotic analysis of SLOPE pattern consistency.

\end{abstract}


\end{frontmatter}

\section{Introduction}

High-dimensional data is currently ubiquitous in many areas of science and industry. Efficient extraction of information from such data sets often requires dimensionality reduction based on identifying the low-dimensional structure behind the data generation process. In this article we focus on a particular statistical model describing the data: the  linear regression model
\begin{equation}\label{eq:regmodel}
	\Y=\X\beta+\beps,
\end{equation}
where $Y\in \R^n$ is a vector of responses,  $X\in \R^{n\times p}$ is a design matrix, $\beta\in \R^p$ is an unknown vector 
of regression coefficients and $\varepsilon\in \R^n$ is a random noise.

It is well-known that  the classical least squares estimator of $\beta$ is BLUE (the best linear unbiased estimator) when the design matrix $X$ is of full column rank. However, it is also well-known that this estimator often exhibits a large variance and a large mean squared estimation error, especially when $p$ is large or when the columns of $X$ are strongly correlated. Moreover, it is not uniquely determined when $p>n$. Therefore, scientists often resort to the penalized least squares estimators of the form,
\begin{equation}\label{eq:pen}
	\hat \beta=\argmin_{b\in \R^p} \left\{\|\Y-\X b\|_2^2+ C \,\mathrm{pen} (b)\right\},
\end{equation}
where $C>0$ and $\mathrm{pen}$ is the penalty on the model complexity. Typical examples of the penalties include $\mathrm{pen}(\beta)=\ell_0(\beta)=\#\{i\colon\beta_i\neq 0\}$, which appears in popular model selection criteria such as AIC \citep{Aka}, BIC \citep{Schw}, RIC \citep{RIC}, mBIC \citep{mBIC} or EBIC \citep{CC08}, or the $\ell_2$ or $\ell_1$ norms, resulting in famous ridge \citep{ridge1, ridge2} or LASSO \citep{chen1994basis,tibshirani1996regression} estimators.
In cases where the penalty function is not differentiable, penalized estimators usually possess
the dimensionality reduction properties as illustrated \textit{e.g.} in \citep{vaiter}. For instance,  
LASSO   may yield some zero components \citep{osborne2000lasso,Tibshirani13}
and thus its dimensionality reduction property is straightforward:  elimination of irrelevant predictors. 

However, in a variety of applications one is interested not only in eliminating variables which are not important but also in merging similar values of regression coefficients. The  prominent statistical example is the multiple regression with categorical variables at many levels, where one may substantially reduce the model dimension and improve the estimation and prediction properties by merging regression coefficients corresponding to ``similar'' levels (see \textit{e.g.} \citep{CASANOVA,  Bayes2, DMR, Bayes1, SCOPE}). Another well-known example of advantages resulting from merging different model parameters are modern Convolutional Neural Networks (CNN), where the ``parameter sharing'' has allowed to ``dramatically lower the number of unique model parameters and to significantly increase network sizes without requiring a corresponding increase in training data'' \cite{DLbook}.

In this article, we explore the dimensionality reduction properties of the well-known convex optimization method, the Sorted L-One Penalized Estimator (SLOPE) \citep{SLOPE1, bogdan2015slope, zeng2014decreasing}. SLOPE has gained considerable attention due to its rich statistical properties (see, \textit{e.g.}, \citep{bogdan2015slope, geneSLOPE, grpSLOPE, Kos2020} for false discovery rate control under various settings, and \citep{Felix, Bellec, su2016} for results on the minimax rates of estimation and prediction).

Following \citep{SLOPE1, bogdan2015slope}, we define the SLOPE estimator as the solution to the optimization problem
\begin{equation}\label{eq:SLOPE}
    \min_{b \in \mathbb{R}^p} \left\{ \frac{1}{2} \|Y - X b\|_2^2 + \sum_{i=1}^p \lambda_i |b|_{(i)} \right\},
\end{equation}
where $|b|_{(1)} \ge |b|_{(2)} \ge \dots \ge |b|_{(p)}$ denote the absolute values of the components of $b$ sorted in nonincreasing order, and $\Lambda = (\lambda_1, \dots, \lambda_p)'$ is a sequence of tuning parameters satisfying $\lambda_1 > 0$ and $\lambda_1 \ge \lambda_2 \ge \dots \ge \lambda_p \ge 0$.

 The SLOPE estimator is arguably the most significant penalized estimator developed in recent years. It can be viewed as an extension of the Octagonal Shrinkage and Clustering Algorithm for Regression (OSCAR) \citep{bondell2008simultaneous}, where the tuning parameter $\Lambda$ has components that decrease arithmetically. It is also closely related to the Pairwise Absolute Clustering and Sparsity (PACS) method \citep{sharma2013consistent}. In this context, the term ``clustering'' reflects the fact that some components of the OSCAR, PACS, and SLOPE estimators may have the same absolute value, while the terms ``sparsity'' and ``shrinkage'' indicate that some components of these estimators can be exactly zero.

SLOPE is also an extension of LASSO whose penalty term is $\lambda\|\cdot\|_1$ ({\it i.e.}, when $\Lambda=(\lambda,\ldots,\lambda)'$  with $\lambda>0$).
Note that contrarily to SLOPE with a decreasing sequence $\Lambda$,
LASSO does not exhibit clusters.
Clustering and sparsity properties for both OSCAR and SLOPE are intuitively illustrated by drawing the elliptic contour lines of the residual
sum of squares $b \mapsto \|Y - Xb\|_2^2$ (when $\ker(X)=\{0\}$)
together with the balls of the sorted $\ell_1$ norm (see, \textit{e.g.},
Figure~2 in \citep{bondell2008simultaneous}, Figure~1 in
\citep{zeng2014decreasing} or Figure~3 in \citep{kremer2020sparse}). 
Known theoretical properties of SLOPE include its ability to  cluster correlated predictors \citep{bondell2008simultaneous,figueiredo2016ordered},  as well as predictors with a similar influence on the $L_2$ loss function \citep{kremer2019sparse}. Specifically,
when $X$ is orthogonal, SLOPE may also cluster components of $\beta$ equal in absolute value \citep{Skalski_2022}. Therefore, dimensionality reduction properties of SLOPE are due to  elimination of irrelevant predictors and grouping predictors having the same influence on $Y$.  Note that, contrary to fused LASSO \citep{tibshirani2005sparsity}, a cluster for SLOPE does not have, in broad generality, adjacent components. 

The clustering properties of SLOPE offer several advantages. One of the most important is its ability to reduce the problem’s dimensionality from $p$ to the number of clusters, thereby lowering variance and enhancing the stability of the estimator. The practical benefits of these clustering effects have been demonstrated, for example, in \citep{kremer2019sparse}, where SLOPE proved effective for sparse portfolio selection. In this setting, SLOPE regularization not only yields sparse and well-diversified portfolios but also improves out-of-sample performance and reduces trading costs by minimizing portfolio turnover. Unlike LASSO, which encourages sparsity but may inconsistently handle similar assets, SLOPE promotes the grouping of assets with comparable risk–return profiles, resulting in portfolios that are both sparse and structured.

These theoretical and practical strengths highlight the importance of developing a rigorous mathematical foundation for the clustering properties of the SLOPE estimator — a goal this article seeks to accomplish.

The key concept for analyzing the clustering properties of SLOPE is the SLOPE pattern, which was first introduced in \citep{schneider2022geometry}. It allows to describe the structure (sparsity and clusters) induced by SLOPE. The SLOPE pattern extracts from a given vector:
\begin{itemize}
	\item[a)] The sign of each component (positive, negative, or zero), 
	\item[b)] The clusters (\textit{i.e.}, indices of components with equal absolute values),
	\item[c)] The hierarchy among the clusters. 
\end{itemize}
The notion of a SLOPE pattern is stronger 
and substantially more informative
than various other structures, such as the model subspace  \citep{vaiter,vaiter2017model} or the sets of irrelevant or clustered components \citep{sharma2013consistent}. Specifically, two vectors that share the same SLOPE pattern also share the same model subspace and have identical sets of zero components as well as components equal in absolute value.

Note that for a given regression model (\ref{eq:regmodel}) the SLOPE pattern depends on relative scaling of different variables.
In the situations where there are no clear reasons or rules for selection of specific measurement units,  we suggest defining the SLOPE pattern with respect to the standardized design matrix. Note that standardizing explanatory variables is also a standard solution for a similar problem of scale-dependent definition of principal components in PCA.

This article focuses on recovering the 
pattern of $\beta$ by SLOPE.  
From a mathematical perspective, the main result is Theorem \ref{th:PT+}, which specifies two conditions (named positivity and subdifferential conditions)  characterizing  pattern recovery by SLOPE in both noisy and noiseless settings. A byproduct of Theorem \ref{th:PT+} is the SLOPE irrepresentability condition: a necessary and sufficient condition for pattern recovery in the noiseless case.  
The word ``irrepresentability'' is a tribute to works written a decade ago on sign recovery by LASSO \citep{fuchs2004sparse,meinshausen2006high,wainwright2009sharp,zhao,zou}. 
However,  when deriving the irrepresentability condition for SLOPE we developed a substantially different mathematical framework, which paves the way for similar analyses of other penalized estimators. Even in the case of the LASSO (see Remark \ref{rem:LASSO}), the sign recovery characterization provided by Theorem \ref{th:PT+} is new and could simplify the proofs of well-known results regarding the LASSO irrepresentability condition. Furthermore, Theorem \ref{th:PT+} provides a sufficient, though not necessary, condition under which a SLOPE solution shares the same model subspace for the sorted $\ell_1$ norm as $\beta$, and correctly identifies the sets of irrelevant or clustered components of $\beta$.
In this way we strengthen the results of
\citep{vaiter,vaiter2017model,sharma2013consistent}.
Finally, the proposed positivity and subdifferential conditions are crucial in developing an algorithm for computing the solution path of SLOPE \citep{dupuis2024solution} or to study pattern recovery by proximal-thresholded SLOPE \cite[Theorem 2.2]{tardivel2024etude}.

In Theorem \ref{th:bound} we consider a noisy case and  under the open SLOPE irrepresentability condition (a condition slightly stronger than the SLOPE irrepresentability condition) we prove that the probability of pattern recovery by  SLOPE tends to $1$ as soon as $X$ is fixed and gaps between distinct absolute values of $\beta$ diverge to infinity.  Additionally, in Theorems  \ref{th:bound_n} and \ref{th:asymp sufficient} we apply the SLOPE irrepresentability condition to derive results on the asymptotic pattern recovery by SLOPE when the number of variables $p$ is fixed and the sample size $n$ diverges to infinity.

While the SLOPE ability to identify the pattern of the vector of regression coefficients $\beta$ is interesting by itself, the related reduction of model dimension also brings an advantage in terms of precision of $\beta$ estimation. This phenomenon is  illustrated in Figure \ref{fig:introduction}, which presents the difference in  precision of LASSO, Fused LASSO and SLOPE estimators, when some of the regression coefficients are equal to each other. 
In this example $n=100$, $p=200$, and the rows of the design matrix are generated as  independent binary Markov chains, with $\mathbb{P}(X_{i1}=1)=\mathbb{P}(X_{i1}=-1)=0.5$ and $\mathbb{P}(X_{i(j+1)}\neq X_{ij})=1-\mathbb{P}(X_{i(j+1)}= X_{ij})=0.0476$.
This value corresponds to the probability of the crossover event between genetic markers spaced every 5 centimorgans and our design matrix can be viewed as an example of 100 independent haplotypes, each resulting from a single meiosis event. In this example, the correlation between columns of the design matrix decays exponentially, $\rho(X_{\cdot i},X_{\cdot j})=0.9048^{|i-j|}$. The design matrix is then standardized, so that each column has a zero mean and a unit variance, and the response variable is generated according to the linear model \eqref{eq:regmodel} with $\beta_1=\ldots=\beta_{30}=40$, $\beta_{31}=\ldots=\beta_{200}=0$ and $\sigma=5$. 
In this experiment the data matrix $X$ and the regression model are constructed such that the LASSO irrepresentability condition holds.
The tuning parameter for LASSO is selected as the smallest value of $\lambda$ for which LASSO can properly identify the sign of $\beta$.
Similarly, the  tuning parameter $\Lambda$ is designed such that the SLOPE irrepresentability condition holds and $\Lambda$ is multiplied by the smallest constant for which SLOPE properly returns the SLOPE pattern. The selected tuning parameters for LASSO and SLOPE are represented in the left panel of Figure \ref{fig:introduction}. For both LASSO and SLOPE, the proposed tuning parameters are close to the values minimizing the mean squared estimation error. The fused LASSO was performed using the {\it fusedlasso} function from the {\it genlasso} library in {\it R}. The tuning parameters $\lambda$ and $\gamma$ were manually selected, so as to minimize the mean squared estimation error. 
Since in this example all methods properly estimate the null components of $\beta$, the right panel in Figure \ref{fig:introduction} illustrates only the accuracy of the estimation of the nonzero coefficients. 
Here  we can observe that the SLOPE ability to identify the cluster structure leads to superior estimation properties. 
SLOPE estimates the regression coefficient vector $\beta$ with virtually no error, while the LASSO estimates range roughly between 36 and 44, and the Fused LASSO estimates fall between approximately 38.8 and 41.9. As a result, the squared error of the SLOPE estimate is more than twenty times smaller than that of the Fused LASSO, and over 100 times smaller than that of the LASSO (0.53 vs 13.25 vs 63.4).

\begin{figure}
	\centering
	\includegraphics[scale=0.6]{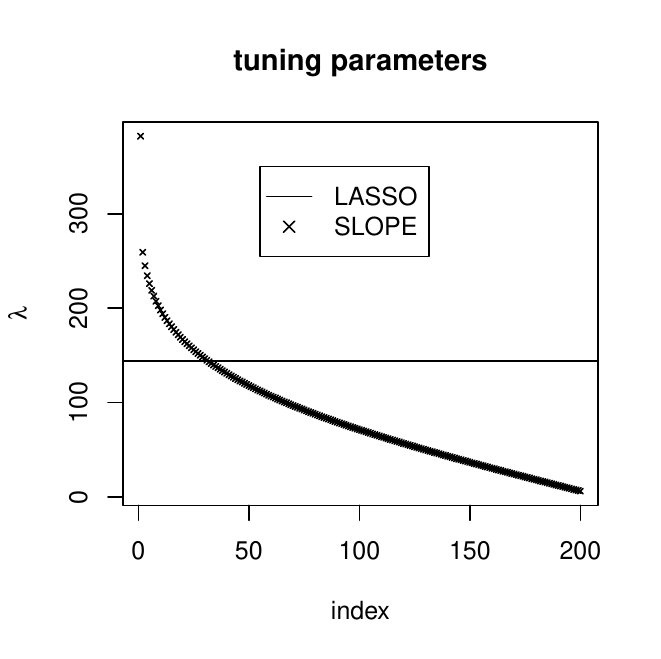}
	\includegraphics[scale=0.6]{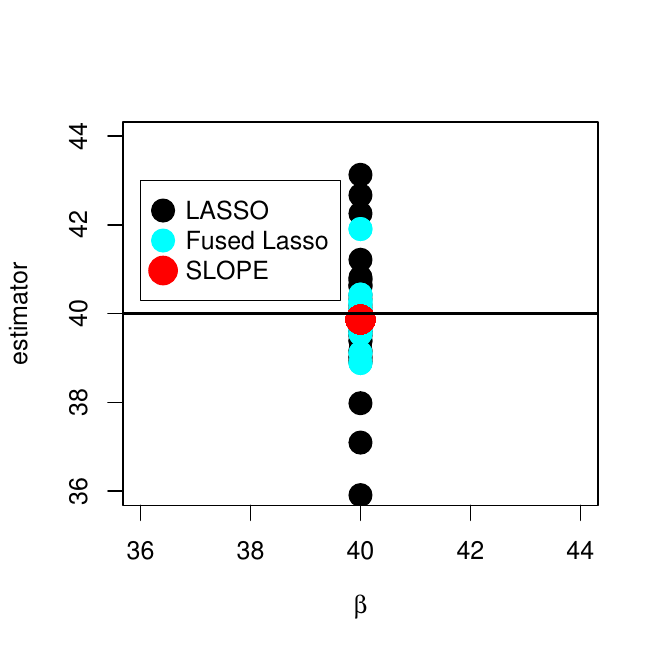}
	\caption{Comparison of LASSO, Fused LASSO and SLOPE when the cluster structure is present in the data. Here $n=100$, $p=200$, the rows of $X$ matrix are simulated as independent binary Markov chains, with the transition probability 0.0476 (corresponding to 5 centimorgans genetic distance). The  correlation between $i^{th}$ and $j^{th}$ column of $X$ decays exponentially as $0.9048^{|i-j|}$.  First $k=30$ columns of $X$ are associated with $Y$ and their nonzero regression coefficients are all equal to 40 (other details are provided in the text).  Left panel represents the value of the tuning parameter for LASSO (solid line) and the sequence of tuning parameters for SLOPE (crosses). The sequences are selected such that both LASSO and SLOPE recover their corresponding patterns with a minimal bias. Right panel represents LASSO, Fused LASSO and SLOPE estimates. The tuning parameters for Fused LASSO were selected manually as to minimize the estimation error.}
	\label{fig:introduction}
\end{figure}
\subsection{Structure of the paper}

 Section \ref{sec:preliminaries} introduces the concept of the SLOPE pattern, which captures the sparsity and clustering structure in the regression coefficients. It also defines key mathematical tools such as the pattern matrix, clustered design matrix, and subdifferential of the sorted $\ell_1$ norm.
 
 Theorem 3.1, the main result of Section \ref{sec:charact}, provides necessary and sufficient conditions for exact SLOPE pattern recovery in both noisy and noiseless settings.
This extends previous work by introducing a ``SLOPE irrepresentability  condition,'' generalizing the LASSO irrepresentability framework.
Even for the LASSO case (see Remark \ref{rem:LASSO}), the sign recovery characterization in Theorem \ref{th:PT+} is new and may simplify existing proofs.

Section \ref{sec:noisy} investigates the conditions under which SLOPE  recovers the true pattern, either as the signal strength increases or as the sample size grows, given appropriate tuning.
It introduces the notion of open irrepresentability, a stronger version of the standard irrepresentability condition, which guarantees asymptotic pattern recovery in high-dimensional settings.

Section \ref{sec:simulations} presents simulation studies that validate the theoretical results from Sections 3 and 4.
The simulations demonstrate that appropriate tuning of the SLOPE penalty yields high probabilities of correct pattern recovery, and they show that SLOPE outperforms LASSO and Fused LASSO when the true regression vector exhibits clustering.

We conclude the paper with a discussion in Section \ref{sec:discussion}. 

The appendix provides detailed  proofs of the main theorems. It also includes auxiliary results such as a law of iterated logarithm for strong consistency and computational verifications of irrepresentability conditions.

\section{Preliminaries and basic notions on clustering properties by SLOPE} \label{sec:preliminaries}
The SLOPE pattern, whose definition is recalled hereafter, is the central notion in this article.
\begin{definition}
	\label{def:SLOPE_pattern}
	Let $b \in \R^p$.  
	The SLOPE pattern of $b$,  $\model(b)$, is defined by
	$$
	\model(b)_i=
	\sign(b_i)\,\mathrm{rank}(|b|)_i, \quad \forall i \in \{1,\dots,p\}
	$$
	where $\mathrm{rank}(|b|)_i\in \{0,1,\dots,k\}$, $k$ is the number of nonzero distinct values in $\{|b_1|,\dots, |b_p|\}$,
	$\mathrm{rank}(|b|)_i=0$ if $b_i=0$,  $\mathrm{rank}(|b|)_i>0$ if $|b_i|>0$ and $\mathrm{rank}(|b|)_i< \mathrm{rank}(|b|)_j$
	if $|b_i|< |b_j|$.
\end{definition}
\ \\

We denote  by $\mathcal{P}^{\rm SLOPE}_p=\model(\R^p)$ the set of SLOPE patterns.
\begin{example}\ \\
	For $a=(4.7,-4.7,0,1.8,4.7,-1.8)'$ we have $\model(a)=(2,-2,0,1,2,-1)'$. For $b=(1.2,-2.3,3.5,1.2,2.3,-3.5)'$ we have $\model(b)=(1,-2,3,1,2,-3)'$.
\end{example}

\begin{definition}	\label{SLOPE pattern matrix}
	Let $0\neq M=(M_1,\ldots,M_p)'\in \mathcal{P}^{\rm SLOPE}_p$ with  $k=\|M\|_\infty$ nonzero clusters.
	The pattern matrix $\U_M\in \R^{p\times k}$ is defined as follows
	\[
	(\U_M)_{ij}=\sign(M_i){\bf 1}_{(|M_i|=k+1-j)},\qquad i \in \{1,\ldots,p\}, \,j \in \{1,\ldots,k\}.
	\]
\end{definition}
Hereafter, the notation $|M|_{\downarrow}=(|M|_{(1)},\ldots,|M|_{(p)})'$ represents the components of $M$ ordered non-increasingly by absolute value.  
\begin{example}
	If $M=(-2,1,0,-1,2)'$, then
	\[
	U_M = \begin{pmatrix}
		-1 & 0 & 0 & 0 & 1 \\ 0 & 1 & 0 & -1 & 0
	\end{pmatrix}' \mbox{ and }U_{|M|_{\downarrow}} = \begin{pmatrix}
		1 & 1 & 0 & 0 & 0 \\ 0 & 0 & 1 & 1 & 0
	\end{pmatrix}'.
	\]
\end{example}
For $k\ge 1$ we denote by
$
\R^{k+}=\{\kappa\in \R^k\colon \kappa_1>\ldots > \kappa_k>0\}.
$
Definition \ref{SLOPE pattern matrix} implies 
that for $0\neq M\in\mathcal{P}^{\rm SLOPE}_p$ and $k=\|M\|_\infty$, for $b\in \R^p$ we have 
$$
\model(b)=M \iff 	\mbox{there exists $\kappa\in\R^{k+}$ such that }b = \U_M \kappa.
$$

\subsection{\texorpdfstring{Clustered  matrix $\tilde X_M$ and clustered parameter $\tilde \Lambda_M$}{Clustered  matrix tildeXM and clustered parameter tildeLambdaM}}

\begin{definition}
	Let  $X\in \R^{n\times p}$, $\Lambda\in \R^{p+}$ and $M\in \mathcal{P}^{\rm SLOPE}_p$.
	The clustered matrix is defined by $\tilde X_M=XU_M$.
	The clustered parameter is defined by $\tilde \Lambda_M= (U_{|M|_{\downarrow}})'\Lambda$. 
\end{definition}

If $M=\model(\beta)$ for $\beta\in\R^p$ satisfies  $\|M\|_\infty<p$, then the pattern $M=(M_1,\ldots,M_p)'$
leads naturally to reduce the dimension of the design matrix $\X$ in the regression problem, by replacing $X$ by
$\tilde{\X}_M$. Actually, if $\model(\beta)=M$, then $\X \beta =\X \U_M\kappa = \tilde{\X}_M \kappa$ for $\kappa\in \R^{k+}$. In particular,
\begin{enumerate}[(i)]
	\item null components $M_i=0$ lead to discard the column $\X_i$ from the design matrix $\X$,
	\item a cluster $K\subset \{1,\ldots,p\}$ of $M$ (components of $M$ equal in absolute value) leads to  replacing the  columns
	$(\X_i)_{i\in K}$ by one column equal to the signed sum: $\sum\limits_{i\in K}\sign(M_i)\X_i$. 
\end{enumerate}

\begin{example}
	Let $X=(X_1|X_2|X_3|X_4|X_5)$, $M=(1,2,-2,0,1)'$ and $\Lambda=(\lambda_1,\lambda_2,\lambda_3,\lambda_4,\lambda_5)'\in \R^{5+}$. Then the clustered matrix and the clustered parameter are given hereafter:
	$$
	\tilde{X}_M = (X_2-X_3|X_1+X_5) \mbox{ and }
	\tilde \Lambda_M = \begin{pmatrix}
		\lambda_1+\lambda_2\\
		\lambda_3+\lambda_4
	\end{pmatrix}.
	$$
\end{example}

\subsection{\texorpdfstring{Sorted $\ell_1$ norm, dual sorted $\ell_1$ norm and  subdifferential}{Sorted L1 norm, dual sorted L1 norm and  subdifferential}}

The sorted $\ell_1$ norm is defined as follows:
\[ 
J_\Lambda(b)=\sum_{i=1}^{p}\lambda_i|b|_{(i)},\quad  b \in \R^p,
\]
where $|b|_{(1)}\ge \ldots \ge |b|_{(p)}$ are the sorted components of $b$ with respect to the absolute value. Given a norm  $\|\cdot\|$ on $\R^p$, we recall that the dual norm $\|\cdot\|^*$    is defined by $\|b\|^*=\max\{v'b\colon\|v\|\le 1\}$ for any $b\in \R^p$. In particular, 
the dual sorted $\ell_1$ norm has an explicit expression given in \citep{negrinho2014orbit} and recalled hereafter:  
\begin{equation}
    \label{def:dual_norm}J^*_\Lambda(b)=\max\left\{\frac{|b|_{(1)}}{\lambda_1},\frac{\sum_{i=1}^{2}|b|_{(i)}}{\sum_{i=1}^{2}\lambda_i},\ldots,\frac{\sum_{i=1}^{p}|b|_{(i)}}{\sum_{i=1}^{p}\lambda_i}\right\},\quad  b\in \R^p.
\end{equation} 
Related to the dual norm,  the subdifferential of a norm $\|\cdot\|$ at $b$ is recalled below (see {\it e.g.} \citep{hiriart2004fundamentals} pages 167 and 180)
\begin{eqnarray}\partial{\|\cdot\|}(b)
	\nonumber	&=&\{v\in \R^p\colon \|z\| \ge \|b\|+v'(z-b)\;\, \forall\, z\in \R^p\}, \\
	\label{eq:subdiff_norm}	&=& \{v\in \R^p\colon \|v\|^*\le 1 \mbox{ and }v'b=\|b\|\}. 
\end{eqnarray}
For the sorted $\ell_1$ norm, 
geometric descriptions of the  subdifferential at  $b\in \R^p$ have been given  in the particular case where  $b_1\ge \ldots \ge b_p\ge 0$
\citep{XDpt,schneider2022geometry,tardivel2020simple}. Hereafter, for an arbitrary $b\in \R^p$,
Proposition \ref{prop:affine} provides a new and useful  formula for the subdifferential of the sorted $\ell_1$ norm. This representation is the crux of the mathematical content of the present paper. 
\begin{proposition}
	\label{prop:affine}
	Let $b\in \R^p$ and $M=\model(b)$. Then we have the following formula:
	\begin{equation}
		\label{eq:aff}
		\partial{J_\Lambda}(b)
		=  \left\{ v\in\R^{p}: J^*_\Lambda(v)\le 1 \mbox{ and } U_{M}' v =  \tilde \Lambda_M \right\}.
	\end{equation}
\end{proposition}

In Proposition \ref{pro:simple} we derive a simple characterization of elements in $\partial{J_\Lambda}(b)$.
The notion of SLOPE pattern is  related  to the subdifferential via the following result.
\begin{proposition}
	\label{prop:subdiff_char} Let $\Lambda=(\lambda_1,\ldots,\lambda_p)'$ where $\lambda_1>\ldots>\lambda_p>0$ and 
	$a,b\in\R^p$. We have    $\model(a)=\model(b)$ if and only if
	$\partial{J_\Lambda}(a)= \partial{J_\Lambda}(b)$.
\end{proposition}
A  proof of Proposition
\ref{prop:subdiff_char} can be found in \citep{schneider2022geometry}. In the Appendix, we provide an independent proof, which is based on Proposition \ref{prop:affine}.

From now on, to comply with Proposition \ref{prop:subdiff_char}, we assume that the tuning parameter  $\Lambda=(\lambda_1,\ldots,\lambda_p)'$ satisfies  
\[
\lambda_1>\ldots>\lambda_p>0.
\]

\subsection{Characterization of  SLOPE minimizers}
The SLOPE estimator is a minimizer of the following optimization problem:
\begin{equation}
	\label{eq:SLOPE_problem}
	S_{X,\Lambda}(Y)=\argmin_{b\in \R^p} \left\{\frac{1}{2}\|Y-Xb\|_2^2+  J_\Lambda(b)\right\}.
\end{equation}
In this article we do not assume that  $S_{X,\Lambda}(Y)$ contains a unique element and potentially $S_{X,\Lambda}(Y)$ can be a non-trivial compact and convex set.
Note however that
cases in which  $S_{X,\Lambda}(Y)$ is not a singleton are very rare. Indeed, the set of matrices $X\in \R^{n\times p}$ for which there exists a $Y\in \R^n$ where  $S_{X,\Lambda}(Y)$ is not a singleton has a null Lebesgue measure on $\R^{n\times p}$ \citep{schneider2022geometry}.  
If $\mathrm{ker}(\X)=\{0\}$, then $\Snew{\X}{\bLambda}{\Y}$ consists of one element.
Recall that a convex function $f$ attains its minimum at a point $b$ if and only if $0\in\partial f(b)$. Since $\partial \frac{1}{2}\|Y-Xb\|_2^2 = \{-X'(Y-Xb)\}$, the SLOPE  estimator
satisfies the following characterization:
$$
\hat\beta\in S_{X,\Lambda}(Y)\quad\Leftrightarrow\quad
X'(Y-X\hat\beta)\in\partial{J_\Lambda}(\hat \beta).
$$

\section{Characterization of pattern recovery by SLOPE} \label{sec:charact}
The characterization of  pattern recovery by  SLOPE given in Theorem
\ref{th:PT+} is a crucial result in this article. 
We recall that $\tilde{\P}_M=(\tilde \X_M')^+\tilde \X_M'=\tilde \X_M \tilde \X_M^+$ is the orthogonal projection  onto $\col(\tilde \X_M)$, where  $A^+$ represents  the Moore-Penrose pseudo-inverse of the matrix
$A$ (see {\it e.g.} \citep{inverses}).

\begin{theorem}\label{th:PT+}
	Let $X\in \R^{n\times p}$, $0\neq \beta\in \R^p$,   $Y=X\beta+\varepsilon$ for $\varepsilon \in \R^n$, $\Lambda\in \R^{p+}$. Let $M=\model(\beta)\in \mathcal{P}^{\rm SLOPE}_p$  and  $k=\|M\|_\infty$.
	Define 
	\begin{align}\label{eq:pi}
		\pi= \X'(\tilde \X_M')^+\tilde\bLambda_M+\X'(\I_n-\tilde{\P}_M)Y.
	\end{align}
	
	There exists $\hat \beta\in S_{X,\Lambda}(Y)$ with $\model(\hat{\beta})=\model(\beta)$ if and only if the two conditions below hold true:
	\[
	\begin{cases}
		\mbox{there exists $s\in\R^{k+}$ such that }
		\tilde \X_M'\Y-  \tilde \bLambda_M = \tilde\X_M'\tilde \X_M s,
		 \mbox{ (positivity condition)}
		\\
		\pi\in\partial{J_\Lambda}(M).   
  \hfill\mbox{ (subdifferential condition)}
	\end{cases}
	\]
	If the positivity and subdifferential conditions are satisfied, then $\hat \beta=U_M s\in S_{X,\Lambda}(Y)$ and $\pi=\X'(\Y-X\hat \beta)$.
\end{theorem}

\begin{remark}
	\leavevmode
	\makeatletter
	\@nobreaktrue
	\makeatother
	\begin{enumerate}[(i)]
		\item 
		When $X$ is  deterministic and $\varepsilon$ has a $\mathrm{N}(0,\sigma^2\I_n)$ distribution, then the event of pattern recovery by SLOPE is the intersection of  statistically independent events:
		\begin{align*}
			A &= \left\{\omega\in \Omega\colon \mbox{there exists $s\in\R^{k+}$ such that }\tilde \X_M'\Y(\omega)-  \tilde \bLambda_M = \tilde\X_M'\tilde \X_M s\right\}, \\
			B &= \left\{\omega\in \Omega\colon \pi(\omega)\in \partial{J_\Lambda}(M)\right\}.
		\end{align*}
		Indeed, since $\tilde{X}'_M = \tilde{X}'_M \tilde{P}_M$ then $\tilde{X}'_MY(\omega)$ depends on $\beps_A(\omega)=\tilde{P}_M \beps(\omega)$. Moreover, $\pi(\omega)$  depends on $\beps_B(\omega)=(\I_n-\tilde{P}_M)\beps(\omega)$. 
		Since $\tilde{P}_M$ is an orthogonal projection, $\beps_A$ and $\beps_B$ have a null covariance matrix. But
		$\varepsilon$ is Gaussian and hence $\beps_A$ and $\beps_B$ are independent. Therefore  events $A$ and $B$ are independent.
		\item Under the positivity condition, the subdifferential condition is equivalent to $J_\Lambda^*(\pi)\le 1$.
		Indeed,  observe that $\tilde \Lambda_M\in \col(\tilde X_M')$ (or equivalently, $\tilde X_M'(\tilde X_M')^+\tilde \Lambda_M=\tilde \Lambda_M$) is necessary for the positivity condition. 
		In view of \eqref{eq:aff}, using the definition of $\pi$, we see that $U_M' \pi=\tilde \Lambda_M$ is equivalent to $\tilde X_M'(\tilde X_M')^+\tilde \Lambda_M=\tilde \Lambda_M$. This follows from the fact that $\tilde{\P}_M$ is the projection matrix onto the vector subspace $\col(\tilde{\X}_M)$, and thus $0'=[(I_n-\tilde{\P}_M)\tilde{\X}_M]'=U_M'X'(I_n-\tilde{P}_M)$.
		\item   The assertion of Theorem \ref{th:PT+} cannot be strengthened. Indeed, if $S_{X,\Lambda}(\Y)$ contains more than one element, then two different minimizers may have different SLOPE patterns.
	\end{enumerate}
\end{remark}
Even though many theoretical properties on sign recovery by LASSO are known (see {\it e.g.} \citep{wainwright2009sharp}),
we believe that it is relevant to give a characterization of sign recovery by LASSO  similar to the characterization of pattern recovery by SLOPE given in Theorem \ref{th:PT+}. 
Such a characterization could simplify proofs of well-known results on LASSO irrepresentability condition.
\begin{remark}\label{rem:LASSO}
	Let $0\neq S\in\{-1,0,1\}^p$ and $k=\|S\|_1$ ($k$ is the number of nonzero components of $S$).
	The signed matrix $\U_S\in \R^{p\times k}$
	is defined by $U_S=({\rm diag}(S))_{\mathrm{supp}(S)}$ where 
	${\rm diag}(S)\in \R^{p\times p}$ is a diagonal matrix and  $({\rm diag}(S))_{\mathrm{supp}(S)}$ denotes the submatrix of ${\rm diag}(S)$ obtained by keeping columns corresponding to indices in $\mathrm{supp}(S)$.
	Observe that for any $0\neq\beta\in\R^p$ there exists a unique $S\in\{-1,0,1\}^p$ and a unique $\kappa_0\in (0,\infty)^{k}$ such that $\beta=U_S\kappa_0$. 
	Define the reduced matrix $\tilde{\X}_S$ and reduced parameter
	$\tilde{\lambda}_S$
	by 
	\[
	\tilde{\X}_S=\X \U_S\mbox{ and } \tilde{\lambda}_S = \lambda 1_k, \mbox{ where }1_k=(1,\ldots,1)'\in \R^k.
	\]
    Moreover, for $b=(b_1,\ldots,b_p) \in \mathbb{R}^p$ define $\sign(b)=(\sign(b_1),\ldots,\sign(b_p))$. 
	Similarly to the proof of Theorem \ref{th:PT+}, one may prove that the necessary and sufficient conditions for the LASSO sign recovery (\textit{i.e.}, the existence of estimator $\betaLASSO$ such that $\sign(\betaLASSO)=\sign(\beta)=S$) are the following 
	\[
	\begin{cases}
		\mbox{there exists $\kappa\in \R_+^k$ such that }\tilde{\X}'_S \Y-\tilde\lambda_S = \tilde\X'_S \tilde\X_S\kappa, 
		\mbox{ (positivity condition)}
		\\
		
		\X'(\tilde\X'_S)^+ 1_k + \frac{1}{\lambda}\X'(\I_n-\tilde\X_S \tilde\X_S^+)Y \in \partial{\|\cdot\|_1}(S). \hfill
		\mbox{ (subdifferential condition)}
	\end{cases}
	\]
	In the noiseless case, when $\varepsilon=0$ and $Y=X\beta$, 
	the subdifferential condition reduces to 
	$\X'(\tilde\X'_S)^+ 1_k\in \partial{\|\cdot\|_1}(S)$
	(or equivalently, $\|\X'(\tilde\X'_S)^+ 1_k\|_\infty\le 1$ and $1_k\in \col(\tilde\X'_S)$). Moreover, when $\ker(X_S)=\{0\}$ then $1_k\in \col(\tilde\X'_S)$ occurs and  $\|\X'(\tilde\X'_S)^+ 1_k\|_\infty\le 1$ is equivalent to $\|\X_{\overline I}'\X_I(\X_I'\X_I)^{-1}S_I\|_\infty\le 1$ where $I={\rm supp}(S)$,  $\overline I=\{1,\ldots,p\}\setminus I$ and  $X_I$ (resp. $X_{\overline I}$) denotes the submatrix of $X$ obtained by keeping columns corresponding to indices in $I$ (resp ${\overline I}$). This latter expression is known as the irrepresentability condition \citep{fuchs2004sparse,zhao,zou}.
\end{remark}

\subsection{SLOPE irrepresentability condition} \label{sec:irrepresentability}

As illustrated by Fuchs \citep{fuchs2004sparse} (Theorem 2), B\"uhlmann and van de Geer \citep{buhlmann2011statistics} (Theorem 7.1) and also recalled in Remark \ref{rem:LASSO}, the irrepresentability condition is related to  sign recovery by LASSO in the noiseless case,
{\textit{i.e.}, when the noise $\varepsilon=0$}. 
Analogously, analyzing pattern recovery by SLOPE in the noiseless case allows to introduce
the SLOPE irrepresentability condition. This   condition will be very useful in the remainder of the article when the noise term $\varepsilon$ is no longer zero. Corollaries 
\ref{cor:mainmr0}
and
\ref{cor:mainmr0bis},
which provide a characterization of pattern recovery by SLOPE in the noiseless case (as defined in \citep{graczyk2023unified}), follow as consequences of  Theorem \ref{th:PT+}.

\begin{corollary}\label{cor:mainmr0}
 Let $X \in \mathbb{R}^{n \times p}$ and $\beta \in \mathbb{R}^p$ where $\model(\beta) = M \neq 0$. In the noiseless case, when $Y=X\beta$
 , the following statements are equivalent:
\begin{enumerate}[(i)]
    \item There exists $\Lambda \in \mathbb{R}^{p+}$ and $\hat{\beta} \in S_{X, \Lambda}(X\beta)$ such that $\model(\hat{\beta}) = \model(\beta)$.
    \item For all $\lambda_1^0>0$, there exists $\Lambda \in \mathbb{R}^{p+}$ with $\lambda_1 < \lambda_1^0$ and $\hat{\beta} \in S_{X, \Lambda}(X\beta)$ such that $\model(\hat{\beta}) = \model(\beta)$.
    \item There exists $\Lambda \in \mathbb{R}^{p+}$ such that $X'(\tilde{X}_M')^{+}\tilde{\Lambda}_M \in \partial J_\Lambda(M)$ (or equivalently $J_\Lambda^*(X'(\tilde{X}_M')^{+}\tilde \Lambda_M)\le 1$ and 
	$\tilde \Lambda_M \in \col(\tilde X_M')$).
    \item For all $\lambda_1^0>0$, there exists $\Lambda \in \mathbb{R}^{p+}$ with $\lambda_1 < \lambda_1^0$ such that $X'(\tilde{X}_M')^{+}\tilde{\Lambda}_M \in \partial J_\Lambda(M)$.
\end{enumerate}
\end{corollary}
Typically, for penalized estimators, the penalty term is scaled by a tuning parameter $\alpha > 0$.
The following corollary addresses the tuning of the SLOPE penalty $J_\Lambda$.

\begin{corollary}\label{cor:mainmr0bis}
When the penalty term $J_\Lambda(\cdot)$, with a fixed $\Lambda\in \R^{p+}$,  is  scaled by a parameter $\alpha>0$, the following statements are equivalent:
\begin{enumerate}
\item[(i)] There exists $\alpha>0$ and $\hat{\beta} \in S_{X, \alpha\Lambda}(X\beta)$ such that $\model(\hat{\beta}) = \model(\beta)$.
\item[(ii)] There exists $\alpha_0>0$ such that for all $\alpha \in (0,\alpha_0)$ there exists $\hat \beta  
	\in S_{X,\alpha\Lambda}(X\beta)$ for which $\model(\hat \beta)=\model(\beta)$.
\item[(iii)]  $X'(\tilde{X}_M')^{+}\tilde \Lambda_M\in \partial{J_\Lambda}(M)$.
\end{enumerate}
\end{corollary}

From now on, given $M=\model(\beta)$, 
we refer to the following inequality and inclusion as the SLOPE irrepresentability condition:
\begin{equation}
	\label{eq:SLOPE_IR}
	J^*_\bLambda\left( \X'(\tilde \X'_M)^+\tilde\bLambda_M\right)\le  1 \mbox{ and } \tilde\bLambda_M\in \col(\tilde \X_M').
\end{equation}

\begin{remark}
	\leavevmode
	\makeatletter
	\@nobreaktrue
	\makeatother
	\begin{enumerate}[(i)]
		\item  When $\ker(\tilde X_M)=\{0\}$, we have   $X'(\tilde{X}_M')^{+}=X'\tilde{X}_M(\tilde{X}_M'\tilde{X}_M)^{-1}$, and consequently,
		the  SLOPE irrepresentability condition  becomes $$J_\Lambda^*(X'\tilde{X}_M(\tilde{X}_M'\tilde{X}_M)^{-1}\tilde \Lambda_M)\le 1.$$ 
		\item   A geometric interpretation of $X'(\tilde{X}'_M)^{+}\tilde{\Lambda}_M$ is provided in the Appendix, see Section \ref{sec:geom}. 
	\end{enumerate}
	
\end{remark}

\begin{example}
	We give two illustrations  in  the particular case where $\Lambda=(4,2)'$, $\beta=(5,0)'$, $\bar \beta=(5,3)'$ and
	$X=(X_1|X_2)\in \R^{n\times 2}$ such that 
	$$X'X=\begin{pmatrix} 1 & 0.6 \\
		0.6 & 1 
	\end{pmatrix}.$$ 
	
	\begin{itemize}
		\item The SLOPE irrepresentability condition does not occur when $\beta=(5,0)'$. Indeed,  $M=\model(\beta)=(1,0)'$, $\tilde X_M=X_1$ (thus $\tilde X_M'\tilde X_M=1$) and $\tilde \Lambda_M=\lambda_1=4$. Therefore 
		$$J_\Lambda^*(X'(\tilde{X}_M')^{+}\tilde \Lambda_M)=J_\Lambda^*(X'\tilde{X}_M(\tilde{X}_M'\tilde{X}_M)^{-1}\tilde \Lambda_M)=J_\Lambda^*(4X'\tilde X_M)=6.4/6>1.$$
		\item
		The SLOPE irrepresentability condition 
		occurs when $\bar\beta=(5,3)'$. Indeed,  $M=\model(\bar\beta)=(2,1)'$, $\tilde \X_M=\X$ and $\tilde \Lambda_M=\Lambda$. Therefore $\ker(\tilde \X_M)=\{0\}$ and
		$$J_\Lambda^*(\X'(\tilde{\X}_M')^{+}\tilde \Lambda_M)=
		J_\Lambda^*(\X'\X(\X'\X)^{-1}\Lambda)=J_\Lambda^*(\Lambda)=1\le 1.$$
	\end{itemize}
	Figure \ref{fig:OWL_path} confirms graphically that  SLOPE irrepresentability condition does not occur for $\beta$  (resp. occurs for $\bar \beta$). 
	Note that, in this setup, the SLOPE solution  is unique (since $\ker(X)=\{0\}$); we denote by $\hat \beta(\alpha)$ the unique element of $S_{\X,\alpha\Lambda}(X\beta)$ and the SLOPE solution path refers to the function $\alpha \in (0,\infty) \mapsto \hat \beta(\alpha)$.

	\begin{figure}[!ht]
		\centering
		\begin{tabular}{c c}
			\includegraphics[scale=0.42]{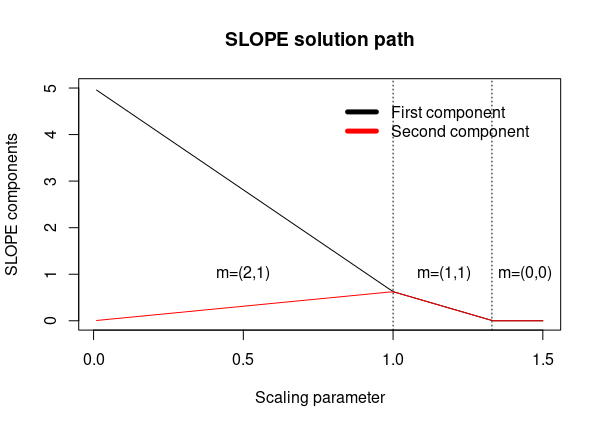}&
			\includegraphics[scale=0.42]{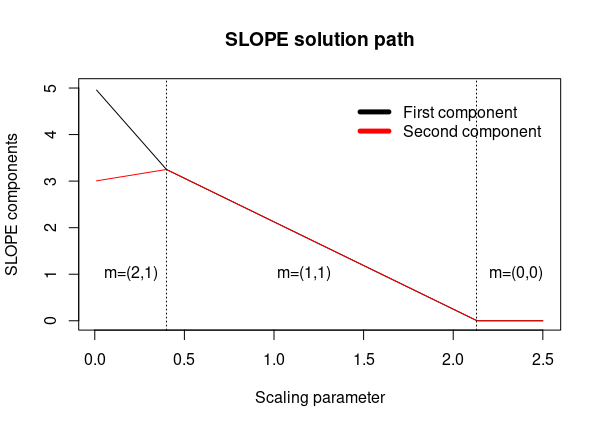}
		\end{tabular}
		\caption{
			On the left, the signal is $\beta=(5,0)'$. Based on this figure one may observe that the pattern of $\beta$ cannot be recovered by SLOPE in the noiseless case. Indeed, for $\alpha\in (0,1)$ we have $\model(\hat\beta(\alpha))=(2,1)'$; when 
			$\alpha\in [1,4/3)$ we have $\model(\hat\beta(\alpha))=(1,1)'$ and when $\alpha>4/3$ then $\hat\beta(\alpha)=0$. Consequently, for every $\alpha>0$ we have  $\model(\hat \beta(\alpha))\neq \model(\beta)=(1,0)'$.
			On the right, the signal is  $\bar\beta=(5,3)'$. Based on this figure one may observe that $\model(\bar\beta)$  is recovered by SLOPE in the noiseless case. Indeed, for $\alpha\in (0,0.4)$ we have $\model(\hat\beta(\alpha))=(2,1)'=\model(\bar\beta)$. 
		}\label{fig:OWL_path}
	\end{figure}
\end{example}

\section{Asymptotic probability on pattern recovery and pattern consistency}\label{sec:noisy}
From now on, in the definition of SLOPE \eqref{eq:SLOPE_problem},  we consider that the penalty term $J_\Lambda(b)$, with fixed $\Lambda\in \R^{p+}$, is multiplied by a scaling parameter $\alpha>0$. We denote by $S_{X,\alpha\Lambda}(Y)$ the set of SLOPE solutions. 
This scaling parameter may, for instance, vary in $(0,\infty)$ for the  solution path, or it can be chosen depending on the standard error of the noise. 
In this section  we consider two asymptotic scenarios and establish conditions on tuning parameters for which the pattern of $\beta$ is recovered. 
In Section \ref{sec:smalln} we consider the case where gaps between distinct absolute values of $\beta$ diverge and in Section \ref{sec:bign} the case where the sample size $n$ diverges.
The proofs rely on Theorem \ref{th:PT+}. We show that the positivity and subdifferential conditions are satisfied under our settings.  It turns out that for the positivity condition the tuning parameter cannot be too large, while for the subdifferential condition it cannot be too small. In this way we consider a tuning parameter of the form $\alpha \Lambda$, where $\Lambda\in\R^{p+}$ is fixed and $\alpha$ varies. We determine the assumptions for the sequence $(\alpha)$ for which both positivity and subdifferential conditions hold true, \textit{i.e.}, for which the pattern is recovered. 

\subsection{\texorpdfstring{$X$ is a fixed matrix}{X is a fixed matrix}}\label{sec:smalln}
The subdifferential condition, given in Theorem \ref{th:PT+}, says that a vector $\pi$ defined in \eqref{eq:pi} belongs to $\partial J_{\alpha \Lambda}(M)$, where $\alpha$ is a scaling parameter. This condition is equivalent to requiring that a vector $\pi_\alpha\ :={\pi}/\alpha$ is an element of $\partial J_{\Lambda}(M)$. We denote the vector $\pi/\alpha$ by
\begin{equation}
	\label{eq:upper_bound_pi}
	\pi_\alpha=X'(\tilde X_M')^+\tilde \Lambda_M+\frac{1}{\alpha}X'(I_n-\tilde P_M)Y=X'(\tilde X_M')^+\tilde \Lambda_M+\frac{1}{\alpha}X'(I_n-\tilde P_M)\varepsilon, \end{equation}
where in the latter equality we have used the fact that $(I_n-\tilde P_M)$ is an orthogonal projection onto $\col(\tilde{X}_M)^\bot$ and therefore $(I_n-\tilde P_M)X\beta = (I_n-\tilde P_M)\tilde{X}_Ms =0$, where $\beta=U_M s$ and $s\in\R^{\|M\|_\infty+}$.

By Theorem \ref{th:PT+}, the probability of  pattern recovery by SLOPE is upper bounded by 
\begin{equation}
	\label{eq:upper_bound}
	\mathbb{P}\left(\exists \hat \beta\in S_{X,\alpha\Lambda}(Y) \mbox{ such that }\model(\hat \beta)=\model(\beta)\right)\le \begin{cases}\mathbb{P}\left(J_\Lambda^*(\pi_\alpha)\le 1 \right),\\
		0 \mbox{ if } \tilde \Lambda_M \notin \col(\tilde X_M'). \end{cases}
\end{equation}
Note that the condition $\tilde \Lambda_M \in \operatorname{col}(\tilde X_M')$ and  $J_\Lambda^*(\pi_\alpha)\le 1$ is necessary  for pattern recovery by SLOPE, but not equivalent to it. Therefore, the inequality in \eqref{eq:upper_bound} is, in general, not an equality.
The first point in Theorem \ref{th:bound} shows that the probability of pattern recovery matches with the upper bound \eqref{eq:upper_bound} when the gaps between the different absolute values of components of $\beta$ are large enough. 
The last point establishes pattern consistency for SLOPE. The formulation of the theorem involves the notion of the relative interior of a set, which we recall below for completeness.

\begin{definition}
The affine hull of a set \( C \subseteq \mathbb{R}^n \), denoted by \( \mathrm{aff}(C) \), is the set of all affine combinations of points in \( C \):
\[
\mathrm{aff}(C) = \left\{ \theta_1 x_1 + \cdots + \theta_k x_k \,\colon\, x_1, \dots, x_k \in C, \, \theta_1 + \cdots + \theta_k = 1 \right\}.
\]
\end{definition}
\bigskip

\begin{definition}
The relative interior of a set  \( C \), denoted \( \mathrm{ri}(C) \), is the interior of \( C \) relative to its affine hull:
\[
\mathrm{ri}(C) = \left\{ x \in C \colon  B(x, r) \cap \mathrm{aff}(C) \subseteq C \text{ for some } r > 0 \right\},
\]
where \( B(x, r) = \{ y \mid \| y - x \| \leq r \} \) is the ball of radius \( r \) centered at \( x \), under any norm \( \| \cdot \| \). All norms define the same relative interior.
\end{definition}


\begin{theorem}
	\label{th:bound}
	Let $X\in \R^{n\times p}$, $0\neq M\in \mathcal{P}_p^{\rm SLOPE}$, and 
	$\Lambda=(\lambda_1,\ldots,\lambda_p)'\in \R^{p+}$. 
	Consider a sequence
	of signals $(\beta^{(r)})_{r\ge 1}$ with pattern $M$:
	\[
	\beta^{(r)}= U_M s^{(r)} \quad \mbox{with}\quad s^{(r)}\in \R^{k+} \mbox{ and } k=\|M\|_\infty,
	\]
	whose strength is increasing in the following sense: 
	\[
	\Delta_r=\min_{1\le i<k}\left(s_i^{(r)}-s_{i+1}^{(r)}\right)\overset{r\to \infty}{\longrightarrow} \infty, \mbox{ with the convention }s_{k+1}^{(r)}=0
	\]
	and let $Y^{(r)}=X\beta^{(r)}+\varepsilon$, where $\varepsilon$ is  a vector in $\R^n$.
	\begin{enumerate}[(i)]
		\item Sharpness of the upper bound: Let $\alpha >0$. 
		If $\varepsilon$ is random, then the upper bound \eqref{eq:upper_bound} is asymptotically reached: 
		\[
		\lim_{r \to \infty}\mathbb{P}\left(\exists \hat \beta \in S_{X,\alpha\Lambda}(Y^{(r)}) \mbox{ such that }\model(\hat \beta)=M\right)=\begin{cases}\mathbb{P}\left(J_\Lambda^*(\pi_\alpha)\le 1 \right),\\
			0 \mbox{ if } \tilde \Lambda_M \notin \col(\tilde X_M'). \end{cases}
		\]
		\item  Pattern  consistency:
		If $\alpha_r\to\infty$, $\alpha_r/\Delta_r\to0$ as $r\to\infty$ and 
		$$X'(\tilde X_M')^+\tilde \Lambda_M\in \mathrm{ri}(\partial{J_\Lambda}(M)),$$ then for any $\beps\in\R^n$ we have
		$$\exists\,r_0>0\,\,\, \forall\, r\ge r_{0}\,\,\, \exists\, \hat \beta \in S_{X,\alpha_r \Lambda}(Y^{(r)}) \mbox{ such that }\model(\hat \beta)=M.$$
	\end{enumerate}
\end{theorem}

\begin{remark}\label{rem:i}
	\begin{enumerate}[(i)]
 \item  The condition $X'(\tilde X_M')^+\tilde \Lambda_M\in \mathrm{ri}(\partial{J_\Lambda}(M))$, called open irrepresentability condition, is slightly stronger than the   irrepresentability condition $X'(\tilde X_M')^+\tilde \Lambda_M\in \partial{J_\Lambda}(M)$.  Note that the tight gap between these conditions is not specific to SLOPE. For instance, for LASSO, the  irrepresentability condition  which is sufficient for support recovery  in the noisy case  is  stronger than the weak irrepresentability condition  for the noiseless case   (see \cite{buhlmann2011statistics} pages 190-192 and 244). 
      \item For the open irrepresentability condition we must check that the cardinality of the set of equalities among the $p$ inequalities corresponding to \(J_\Lambda^*(X'(\tilde X_M')^+\tilde \Lambda_M)\le 1\), see \eqref{def:dual_norm},
      is exactly $\|M\|_\infty$.
      That is,   $X'(\tilde X_M')^+\tilde \Lambda_M\in \mathrm{ri}(\partial{J_\Lambda}(M))$ is equivalent to the following computationally verifiable conditions:
		\begin{equation}
		     \begin{cases}J_\Lambda^*(X'(\tilde X_M')^+\tilde \Lambda_M)\le 1 \mbox{ and }\tilde \Lambda_M\in \col(\tilde X_M'),\\
			\left|\left\{i\in \{1,\ldots,p\}\colon \sum_{j=1}^{i}|X'(\tilde X_M')^+\tilde \Lambda_M|_{(j)}=\sum_{j=1}^{i}\lambda_j\right\}\right|=\|M\|_\infty.\end{cases}\end{equation}
		This equivalence follows from Proposition \ref{pro:simple}.	
      
		\item Let us assume that the distributions of $\varepsilon$ and $-\varepsilon$ are equal. Because the unit ball of the dual sorted $\ell_1$ norm
		is convex,
		when  $J_\Lambda^*(X'(\tilde X_M')^+\tilde \Lambda_M)>1$ then, independently of $\alpha>0$, the probability of pattern recovery is smaller than $1/2$, namely 
		$$\mathbb{P}\left(\exists \hat \beta \in S_{X,\alpha\Lambda}(Y) \mbox{ such that }\model(\hat \beta)=M\right)\le 1/2.$$
		This inequality corroborates Theorem 4.6 in \citep{graczyk2023unified}. For LASSO, a similar inequality on the probability of sign recovery is given in \citep{wainwright2009sharp}. 
		\item     
		In Section~\ref{sec:simulations}, we illustrate that, under the open irrepresentability condition, one may select $\alpha>0$  to fix the asymptotic probability of pattern recovery at a level arbitrarily close to $1$ (a similar result for LASSO is given in \citep{tardivel2022sign}).
	\end{enumerate}
\end{remark}

\subsection{\texorpdfstring{$\X$ is random, $p$ is fixed, $n$ tends to infinity}{X is random, p is fixed, n tends to infinity}}\label{sec:bign}

In this section  we discuss asymptotic properties of the SLOPE estimator in the low-dimensional regression model in which $p$ is  fixed   and the sample size $n$ tends to infinity.

For each $n \ge p$ we consider a linear regression problem 
\begin{equation}\label{reg2}
	\Y_n=\X_n\beta +\beps_n,
\end{equation}
where $\X_n\in\R^{n\times p}$ is a random design matrix.
We now list our assumptions:
\begin{enumerate}
	\item[A.] $\beps_n=(\epsilon_1,\ldots,\epsilon_n)'$, where $(\epsilon_i)_i$ are i.i.d. centered with finite variance.
	\item[B1.] A sequence of design matrices $\X_1, \X_2,\ldots$ satisfies the condition
	\begin{equation} \label{assumption on design}
		\frac{1}{n}\X_n' \X_n \stackrel{\mathbb{P}}{\longrightarrow} \C,
	\end{equation}
	where $\C$ is a deterministic positive definite symmetric $p\times p$ matrix.
	\item[B2.]
	For each $j=1,\ldots,p$,
	\begin{align*}
		\frac{\max_{i=1,\ldots,n} |X_{ij}^{(n)}|}{\sqrt{\sum_{i=1}^n (X_{ij}^{(n)})^2}} \stackrel{\mathbb{P}}{\longrightarrow}0.
	\end{align*}
	\item[C.] $(\X_n)_n$ and $(\epsilon_n)_n$ are independent. 
\end{enumerate}

We will consider a sequence of tuning parameters $(\Lambda_n)_n$ defined by
\begin{align*}
	\bLambda_n=\alpha_n\bLambda,
\end{align*}
where $\bLambda\in\R^{p+}$ is fixed and $(\alpha_n)_n$ is a sequence of positive numbers.

Let $\betaSLOPE_n$ be an element from the set $\Snew{\X_n}{\bLambda_n}{\Y_n}$ of SLOPE minimizers.
Under assumption B1, for large $n$ with high probability, the set $\Snew{\X_n}{\bLambda_n}{\Y_n}$ consists of one element. Indeed, we have
\[
\mathbb{P}\left(\ker(X_n)=\{0\}\right)=\mathbb{P}\left(X_n'X_n \mbox{ is positive definite}\right)\stackrel{n\to\infty}{\longrightarrow}1
\]
and $\ker(\X_n)=\{0\}$ ensures the existence of the unique SLOPE minimizer. 
In a natural setting, the strong consistency of $\betaSLOPE_n$ can be characterized in terms of behaviour of the tuning parameter, see Theorem \ref{thm:consist} or \cite[Th. 4.1]{Skalski_2022}. At this point we note that if \eqref{assumption on design} holds almost surely, then condition $\alpha_n/n\to0$ ensures that $\betaSLOPE_n\stackrel{a.s.}{\longrightarrow}\beta$. Thus, if $\beta$ does not have any clusters nor zeros, \textit{i.e.}, $\|\model(\beta)\|_\infty=p$, then the $\alpha_n/n\to0$ suffices for $\model(\betaSLOPE_n)\stackrel{a.s.}{\longrightarrow} \model(\beta)$. However, if $\|\model(\beta)\|<p$, then 
the situation is more complex as we shall show below.

The first of our asymptotic results concerns the consistency of the pattern recovery by the SLOPE estimator. We note that condition B2 is not necessary for the SLOPE pattern recovery. This assumption was 
introduced to ensure the existence of a Gaussian vector in Theorem \ref{th:bound_n} (i).

The formulation of the following theorem involves the 
notion of the pattern matrix $U_M$, as defined in Definition \ref{SLOPE pattern matrix}.
\begin{theorem}	\label{th:bound_n}
	Under the assumptions A, B1, C, the following statements hold true.
	\begin{enumerate}
		\item[(i)] If $B2$ is additionally satisfied and moreover $\alpha_n=\sqrt{n}$, then  
		\[
		\lim_{n \to \infty}\mathbb{P}\left(\model(\betaSLOPE_n)=\model(\beta)\right)=\mathbb{P}\left( J_\bLambda^*(\Zb)\le 1\right),
		\]
		where $\Zb\sim \mathrm{N}(CU_M(U'_MCU_M)^{-1}\tilde\Lambda_M, \sigma^2[C-CU_M (U'_M CU_M)^{-1}U'_M C])$.
		\item[(ii)] Assume 
		\begin{align}\label{eq:Cri}
			\C\U_M(\U_M'\C\U_M)^{-1}\tilde{\bLambda}_M\in\mathrm{ri}(\partial J_\Lambda(M)).
		\end{align}
		The pattern of SLOPE estimator is consistent, {\it i.e}.
		\[
		\model(\betaSLOPE_n)\stackrel{\mathbb{P}}{\longrightarrow}\model(\beta),
		\]
		if and only if 
		\[
		\lim_{n \to \infty}  \frac{\alpha_n}{n}=0\qquad\mbox{and}\qquad		\lim_{n \to \infty}  \frac{\alpha_n}{\sqrt{n} } = \infty.
		\]
		\item[(iii)] The  condition 
		\begin{align}\label{eq:matrixIR}
			J_\bLambda^*\left(\C\U_M(\U_M'\C\U_M)^{-1}\tilde{\bLambda}_M\right)\le 1
		\end{align} 
		is necessary for pattern consistency of SLOPE estimator.
	\end{enumerate}
\end{theorem}
The random vector $Z$ belongs to the smallest affine space containing $\partial{J_\Lambda}(b)$, {\it i.e.}, $\aff(\partial{J_\Lambda}(b))=\{v\in \R^p\colon U_M'v=\tilde \Lambda_M\}$, see Lemma \ref{lemma:affine_space}.

Condition \eqref{eq:Cri} is the open SLOPE irrepresentability condition in the $n\to\infty$ regime. 
The above result should be compared with \cite[Theorem 1]{zhao}, where the same conditions on the LASSO tuning parameter ensure consistency of sign recovery by the LASSO estimator. Below we make a step further and consider the strong consistency of SLOPE pattern recovery by $\betaSLOPE_n$. 
Although this was not Zhao's and Yu's main focus, it can be deduced from \cite[Theorem 1]{zhao} that if for $c\in(0,1)$ the LASSO tuning parameter $\lambda_n$ satisfies $\lambda_n/n\to 0$ and $\lambda_n/n^{\frac{1+c}{2}}\to\infty$, then under the strong LASSO irrepresentability condition, one has $\sign(\betaLASSO_n)\stackrel{a.s.}{\longrightarrow}\sign(\beta)$.
Even though the patterns are discrete objects, as the underlying probability space is uncountable, the convergence in probability does not imply the almost sure convergence. We show below that if $\alpha_n/n\to 0$ and $\alpha_n/\sqrt{n}\to\infty$, then $\model(\betaSLOPE_n)$ is not strongly consistent and one actually has to impose a slightly stronger condition \eqref{eq:lambdan}.

For the purpose of the a.s. convergence, we strengthen the assumption on design matrices:
\begin{enumerate}
	\item[B'.]  Assume that  the rows of $\X_n$ are independent and that each row of $\X_n$ has the same law as $\xi$, where $\xi$
	is
	a random vector whose components are linearly independent a.s. and that $\E[\xi_i^2]<\infty$ for $i=1,\ldots,p$.
\end{enumerate}

\begin{remark}\label{rem:rem0}
	Under B', by the strong law of large numbers, we have 
	$n^{-1}\X_n' \X_n \stackrel{a.s.}{\longrightarrow} \C$,
	where $\C=(C_{ij})_{ij}$ with $C_{ij} = \E[\xi_i \xi_j]$. Moreover, $\C$ is positive definite if and only if the random variables $(\xi_1,\ldots,\xi_p)$ are linearly independent a.s. Indeed, for $t\in\R^p$ we have
	$t'\C t = \E[(\sum_{i=1}^p t_i \xi_i)^2]>0$
	if and only if $\sum_{i=1}^p t_i \xi_i\neq0$ a.s. for all $t\in\R^p\setminus\{0\}$.
	
	Since B' ensures that \eqref{assumption on design} holds a.s., it also implies that for large $n$, almost surely there exists a unique SLOPE minimizer. We denote this element by $\betaSLOPE_n$.
\end{remark}

\begin{theorem}\label{th:asymp sufficient}
	Under $A$, $B'$ and $C$ assume that a sequence $(\alpha_n)_n$ satisfies
	\begin{align}\label{eq:lambdan}
		\lim_{n \to \infty}  \frac{\alpha_n}{n}=0\qquad\mbox{and}\qquad		\lim_{n \to \infty}  \frac{\alpha_n}{\sqrt{n \log\log n} } = \infty.
	\end{align}
	If \eqref{eq:Cri} holds, 
	then 
	the sequence $(\betaSLOPE_n)_n$  recovers almost surely
	the pattern of $\beta$ asymptotically, \textit{i.e.},
	\begin{equation} \label{property: ASAMRP}
		\model(\betaSLOPE_n)\stackrel{a.s.}{\longrightarrow}\model(\beta).
	\end{equation}
\end{theorem}

\begin{remark}\label{rem:asymp} Assume that \eqref{eq:Cri} is satisfied and set $\alpha_n=c\sqrt{n \log\log n}$ for $c>0$. Then \eqref{eq:lambdan} is not satisfied and with positive probability, the true SLOPE pattern is not recovered. See also Appendix \ref{appB}, where we present more refined results on the strong consistency of the SLOPE pattern. The $\log\log n$ correction in \eqref{eq:lambdan} comes from the law of iterated logarithm.
\end{remark}

\section{Simulation study}\label{sec:simulations}

This simulation study aims at illustrating  Theorems \ref{th:bound} and \ref{th:bound_n}.
Hereafter,  we consider the linear regression model $Y=X\beta+\varepsilon$, where $X\in \R^{n\times p}$ and $\varepsilon\in \R^n$ has i.i.d. $\mathrm{N}(0,1)$ entries. 
	Up to a constant, we choose components of  $\Lambda=(\lambda_1,\ldots,\lambda_p)'$ as expected values of ordered standard Gaussian statistics. Let $Z_{(1)}\ge \ldots \ge Z_{(p)}$ be ordered statistics of i.i.d. $\mathrm{N}(0,1)$ random variables.  
	An approximation of $\E[Z_{(i)}]$ for some $i\in \{1,\ldots,p\}$, denoted $E(i,p)$, is given hereafter (see \citep{harter1961expected} and references therein) 
	$$E(i,p)=-\Phi^{-1}\left(\frac{i-0.375}{p+1-0.750}\right),$$ 
	where $\Phi$ is the cumulative distribution function of an $\mathrm{N}(0,1)$ random variable. We  set 
\begin{equation}\label{lambdasim}
\Lambda=(\lambda_1,\ldots,\lambda_p)\;\;\mbox{with}\;\;\lambda_i=E(i,p)+E(p-1,p)-2E(p,p).
\end{equation}

\subsection{Sharp upper bound when \texorpdfstring{$X$}{X} is orthogonal}\label{sec:simc}

This example  illustrates Theorem \ref{th:bound}, which concerns the limiting probability of pattern recovery as  signal strength tends to infinity. 
 We assume that  $p=100$, $c$ is a positive real number,
$X\in \R^{n\times p}$ is orthogonal ($X'X=I_{100}$), and	$\beta \in \R^p$ is defined as follows:
\begin{equation}\label{beta}
\beta_1=\ldots=\beta_{25}=c,\;\; \beta_{26}=\ldots=\beta_{50}=-c/2,\; \beta_{51}=\ldots=\beta_{100}=0\;.
\end{equation}
To compute the value $\alpha_{0.95}$ of the scaling parameter  for which the upper bound is  $0.95$ we  note that ${\pi}_\alpha$ is a Gaussian vector having a 
 $$\mathrm{N}\left(X' (\tilde X_M')^+\tilde \Lambda_M,\alpha^{-2}X'(I-\tilde X_M \tilde X_M^+)X
\right)$$ distribution. Moreover,  since $M=\model(\beta)$ satisfies: $M_1=\ldots=M_{25}=2$, 
$M_{26}=\ldots=M_{50}=-1$ and $M_{51}=\ldots=M_{100}=0$ we have 
\begin{equation*}
X' (\tilde X_M')^+\tilde \Lambda_M=\mu,
    \end{equation*}
    where $\mu_1=\ldots=\mu_{25}=\frac{1}{25}\sum_{i=1}^{25} \lambda_i$, $\mu_{26}=\ldots=\mu_{50}=-\frac{1}{25}\sum_{i=26}^{50} \lambda_i$,\\ $\mu_{51}=\ldots=\mu_{100}=0$, and
    \begin{equation}
	\nonumber X'(I_n-\tilde X_M \tilde X_M^+)X=\begin{pmatrix} \Sigma & 0&0\\
          0&\Sigma&0\\
		0 &0& I_{p/2}
        \end{pmatrix},
        \end{equation}
    where $\Sigma$ is the matrix of the dimension $p/4\times p/4$ given by

\begin{equation}
    \Sigma=\begin{pmatrix}1-4/p&-4/p&\ldots&-4/p\\
		-4/p & 1-4/p & \ddots & \vdots\\
        \vdots & \ddots & \ddots & -4/p \\
		-4/p & \ldots &  -4/p & 1-4/p
	\end{pmatrix}.
\end{equation}
The matrix $\Sigma$ appears twice in the covariance structure, as both nonzero clusters have the same size of $p/4$.

Since the open SLOPE irrepresentability condition holds, there exists the value $\alpha_{0.95}$ such that
\[
    \mathbb{P}\big(J_\Lambda^*(\pi_{\alpha_{0.95}}) \leq 1\big) = 0.95.
\]

In practice, we simulated $50\,000$ instances of the random vector $Z \sim \mathrm{N}(0, X'(I-\tilde X_M \tilde X_M^+)X)$ and identified the value $\alpha_{0.95} = 9.45$, such that
\[
    \mathbb{P}\Big(J_\Lambda^*\Big(\mu + \tfrac1{\alpha_{0.95}} Z\Big) \le 1\Big) \approx 0.95.
\]
Figure \ref{fig:upper_bound_n_fixed_2} illustrates that indeed  the probability  of pattern recovery in the model \eqref{beta} by SLOPE with a regularizing sequence $9.45\Lambda$ converges to $0.95$ as $c$ increases to infinity.

\begin{figure}[!h]
	\centering
	\includegraphics[scale=0.5]{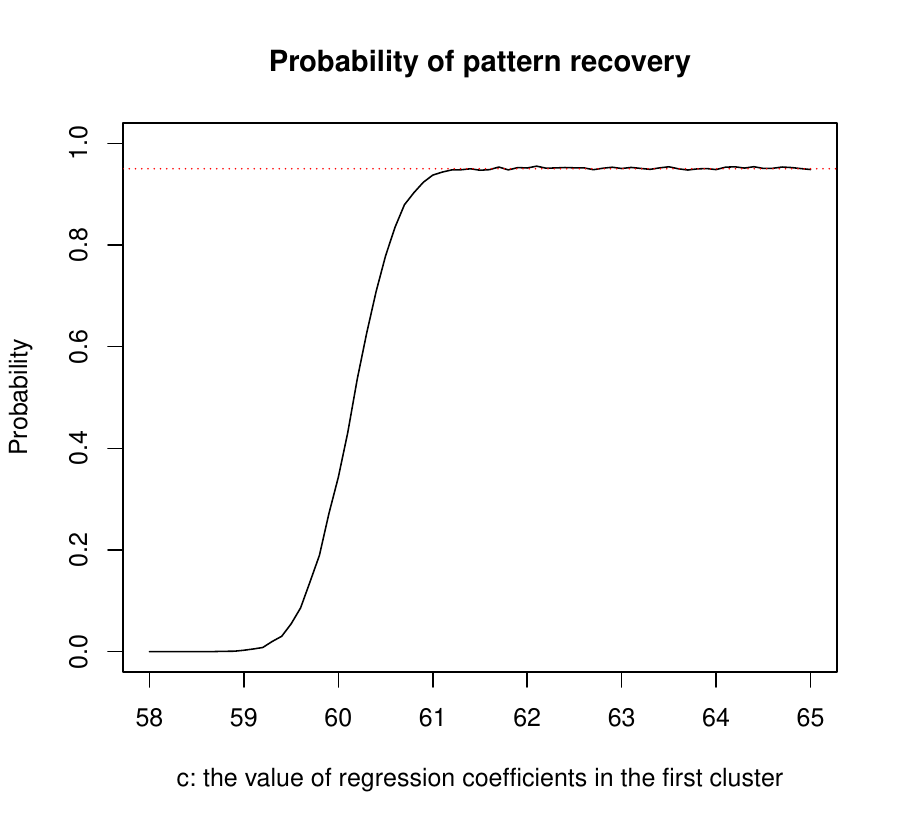}
	\caption{Probability of the pattern recovery in the model \eqref{beta} by SLOPE with a tuning parameter $\alpha\Lambda$, where $\alpha=9.45$ and $\Lambda$ is given in \eqref{lambdasim}.}
\label{fig:upper_bound_n_fixed_2}
\end{figure}

\subsection{\texorpdfstring{Limiting probability when $n\rightarrow\infty$}{Limiting probability when n to infty}}\label{Sec:n}

In this section, we illustrate Theorem~\ref{th:bound_n}, which describes the asymptotic performance of SLOPE as the sample size $n$ tends to infinity. We consider a setting in which both the predictors and the regression coefficients exhibit a clustered structure.

Specifically, we assume that the design covariance matrix for $p = 100$ regressors is block-diagonal:
\begin{equation}
C = \begin{pmatrix} 
    \Sigma & 0 & 0 & 0 \\
    0 & \Sigma & 0 & 0 \\
	0 & 0 & \Sigma & 0 \\
    0 & 0 & 0 & \Sigma  
        \end{pmatrix}
\end{equation}
where $\Sigma$ is a compound symmetry matrix of dimension $25 \times 25$, with $\Sigma_{i,i} = 1$ and $\Sigma_{i,j} = 0.8$ for $i \neq j$.

The true regression coefficient vector $\beta \in \mathbb{R}^p$ is defined as:
\begin{equation*}
\beta_1 = \ldots = \beta_{25} = 30,\quad 
\beta_{26} = \ldots = \beta_{50} = -30,\quad 
\beta_{51} = \ldots = \beta_{100} = 0.
\end{equation*}
Thus, in this example, the pattern $M = \model(\beta)$ satisfies:
\[
M_1 = \ldots = M_{25} = 1,\quad 
M_{26} = \ldots = M_{50} = -1,\quad 
M_{51} = \ldots = M_{100} = 0.
\]
Furthermore, we have:
\begin{equation*}
 CU_M (U_M' C U_M)^{-1} \tilde{\Lambda}_M=\mu,
\end{equation*}
where 
\[
\mu_1 = \ldots = \mu_{25} = \frac{1}{50} \sum_{i=1}^{50} \lambda_i,\quad
\mu_{26} = \ldots = \mu_{50} = -\mu_1,\quad
\mu_{51} = \ldots = \mu_{100} = 0.
\]
Finally, the covariance matrix of the vector $Z$ in Theorem \ref{th:bound_n} takes the form:
\begin{equation}
\Sigma_Z= C - C U_M (U_M' C U_M)^{-1} U_M' C =
\begin{pmatrix} 
\Sigma-U & U & 0 & 0\\
U & \Sigma - U & 0 & 0 \\
0 & 0 & \Sigma & 0 \\
0 & 0 & 0 & \Sigma
\end{pmatrix},
\end{equation}
where $U$ is the $25 \times 25$ matrix in which all entries are equal to $(1 + 24 \cdot 0.8)/50$.

By simulating 50\,000 instances of the multivariate normal vector $Z\sim \mathrm{N}(0, \Sigma_Z) $ we found the value $\alpha=2.89$, such that $\mathbb{P}\left(J_{\Lambda}^{\star}\left(\mu+\frac{1}{\alpha}Z\right)\leq 1\right) \approx 0.95$

According to Theorem \ref{th:bound_n}  SLOPE with the tuning sequence $2.89\Lambda\sqrt{n}$ -- where $\Lambda$ is specified in (\ref{lambdasim}) -- should recover the true pattern with the probability 0.95 as $n \to \infty$. This phenomenon is illustrated in Figure \ref{fig:patt_recovery_n_to_infty}, where the probability of the pattern recovery stabilizes at 0.95 for $n\geq 1500$.  
\begin{figure}[h!]
\centering
\includegraphics[scale=1]{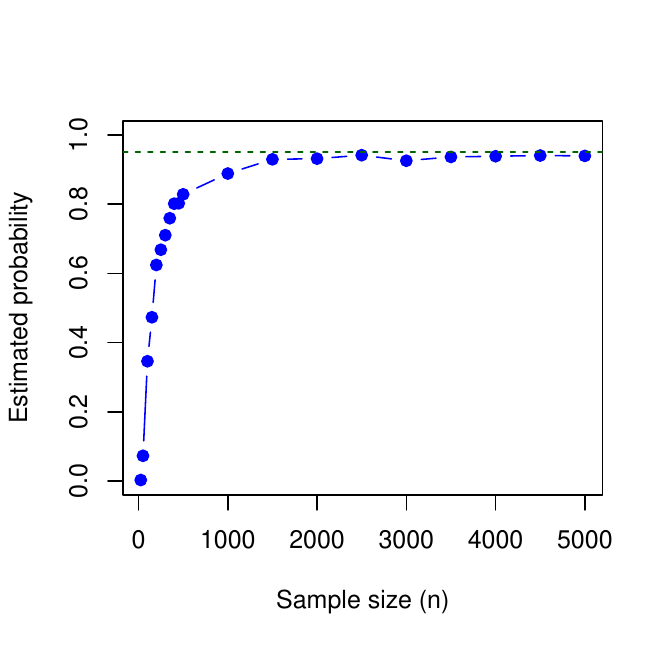}
\caption{Estimates of probability of pattern recovery by SLOPE as a function of $n$. The scaling parameter $\alpha_{0.95}=2.89$ is chosen to fix the limiting probability of pattern recovery at $0.95$. }
\label{fig:patt_recovery_n_to_infty}
\end{figure}

Additionally, Figure~\ref{fig:RMSE_n_to_infty} illustrates the root mean square error (RMSE) of the SLOPE estimator and compares it to the RMSE of both the ordinary least squares (LS) estimator and the debiased SLOPE estimator. The latter is obtained by performing a least squares fit using the reduced model selected by SLOPE, \textit{i.e.}, using the design matrix $\tilde{X} = XU_{\hat{M}}$.

As shown in the figure, SLOPE consistently outperforms LS in terms of RMSE in this example. Moreover, the estimation accuracy can be further improved by debiasing SLOPE -- specifically, by applying least squares estimation within the reduced model. In the setting considered here, this two-stage version of SLOPE achieves near-perfect performance, with a negligible estimation error.

\begin{figure}[h!]
\centering
\includegraphics[scale=1]{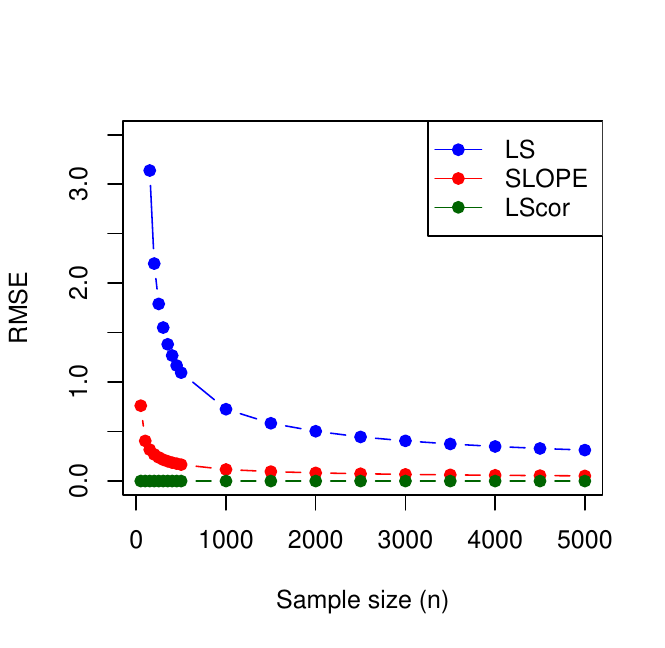}
\caption{Estimated Root Mean Squares of the estimators obtained by Least Squares (LS), SLOPE with the sequence of parameters as in Figure \ref{fig:patt_recovery_n_to_infty} (SLOPE) and the Least Squares estimators using SLOPE dimensionality reduction, \textit{i.e.}, $\tilde X=XU_{\hat M}$ (LScor).}
\label{fig:RMSE_n_to_infty}
\end{figure}

\subsection{Small \texorpdfstring{$n$}{n} performance}

In Figure \ref{fig:Est_n_small} we present a representative example of the performance of SLOPE for the setting from Section \ref{Sec:n} and a very small $n=25$. We compare SLOPE to Fused LASSO, since the coefficients are clustered according to the proximity of variables in the design matrix. For SLOPE we use the sequence of tuning parameters as proposed in Section \ref{Sec:n}, while the Fused LASSO is tuned manually to minimize the estimation error. We used the {\it fusedlasso} function from the {\it genlasso} library and manually selected $\gamma=0.25$ as the value for which we obtained the smallest RMSE over the range of $\lambda$ values automatically proposed by the {\it fusedlasso} algorithm. We can observe that while SLOPE cannot precisely estimate the pattern, shrinkage towards the common absolute mean in both clusters allows for obtaining a very precise estimation of $\beta$, which is substantially more accurate than the fused lasso estimator. We believe that this is due to the fact that SLOPE effectively shrinks both clusters towards the same absolute value, while fused LASSO does not have this advantage. 

\begin{figure}[h!]
\centering
\includegraphics[scale=1]{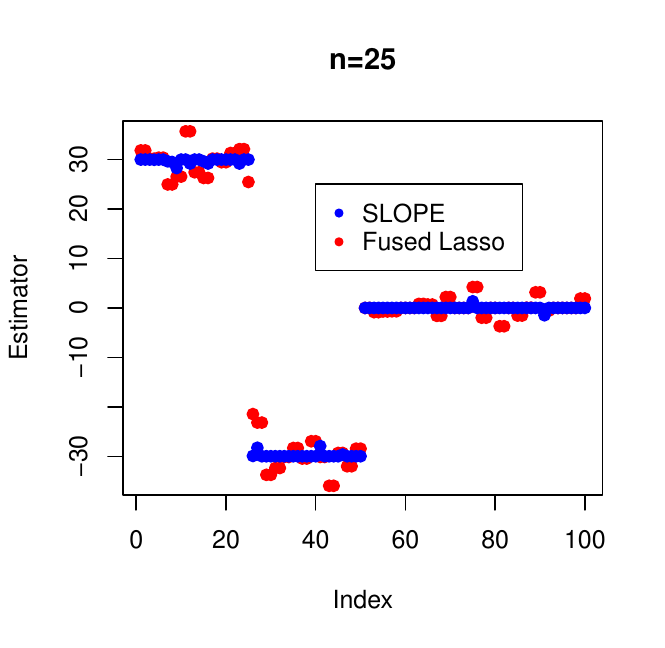}
\caption{Comparison of SLOPE and Fused LASSO estimators for $n = 25$ and $p = 100$. Among the $100$ regression coefficients, $50$ are nonzero: the first $25$ are equal to $30$, and the next $25$ are equal to $-30$.
}
\label{fig:Est_n_small}
\end{figure}

\section{Discussion} \label{sec:discussion}
In this article we make an important step in understanding the clustering properties of SLOPE and we have shown that the irrepresentability condition provides theoretical guarantees for SLOPE pattern recovery.
However, this by no means closes the topic of the SLOPE pattern recovery. Similarly to the irrepresentability condition for LASSO, the SLOPE irrepresentability condition is rather stringent and imposes a strict restriction on the number of nonzero clusters in $\beta$. On the other hand, in \citep{tardivel2022sign} it is shown that a much weaker condition for LASSO is required to separate the estimators of the null components of $\beta$ from the estimators of nonzero regression coefficients. This  condition, called accessibility (also called identifiability), requires that the vector $\beta$ has the minimal $\ell_1$ norm among all vectors $\gamma$ such that $X\beta=X\gamma$. Thus, when the accessibility condition is satisfied one can recover the sign of $\beta$ by thresholding LASSO estimates.
Empirical results from \citep{tardivel2022sign} suggest that this weaker condition is also sufficient for the sign recovery by the adaptive LASSO \citep{zou}.  In this case
rescaling the design matrix according to the initial estimates of regression coefficients modifies
the original irrepresentability condition, so it can be satisfied for a given specific true
sign vector of regression coefficients.    In the recent article \citep{graczyk2023unified} it is shown that a similar result holds for SLOPE, whose accessibility condition holds if the vector $\beta$ has the smallest sorted $\ell_1$ norm among all vectors $\gamma$ such that  $X\beta=X\gamma$. 
In \citep{graczyk2023unified} or in \cite[Theorem 2.2]{tardivel2024etude} it is shown that when the accessibility condition is satisfied then applying the proximal operator of the sorted $\ell_1$ norm to SLOPE allows to recover the pattern of the regression coefficients.
\begin{figure}[!ht]
\centering
\includegraphics[scale=0.55]{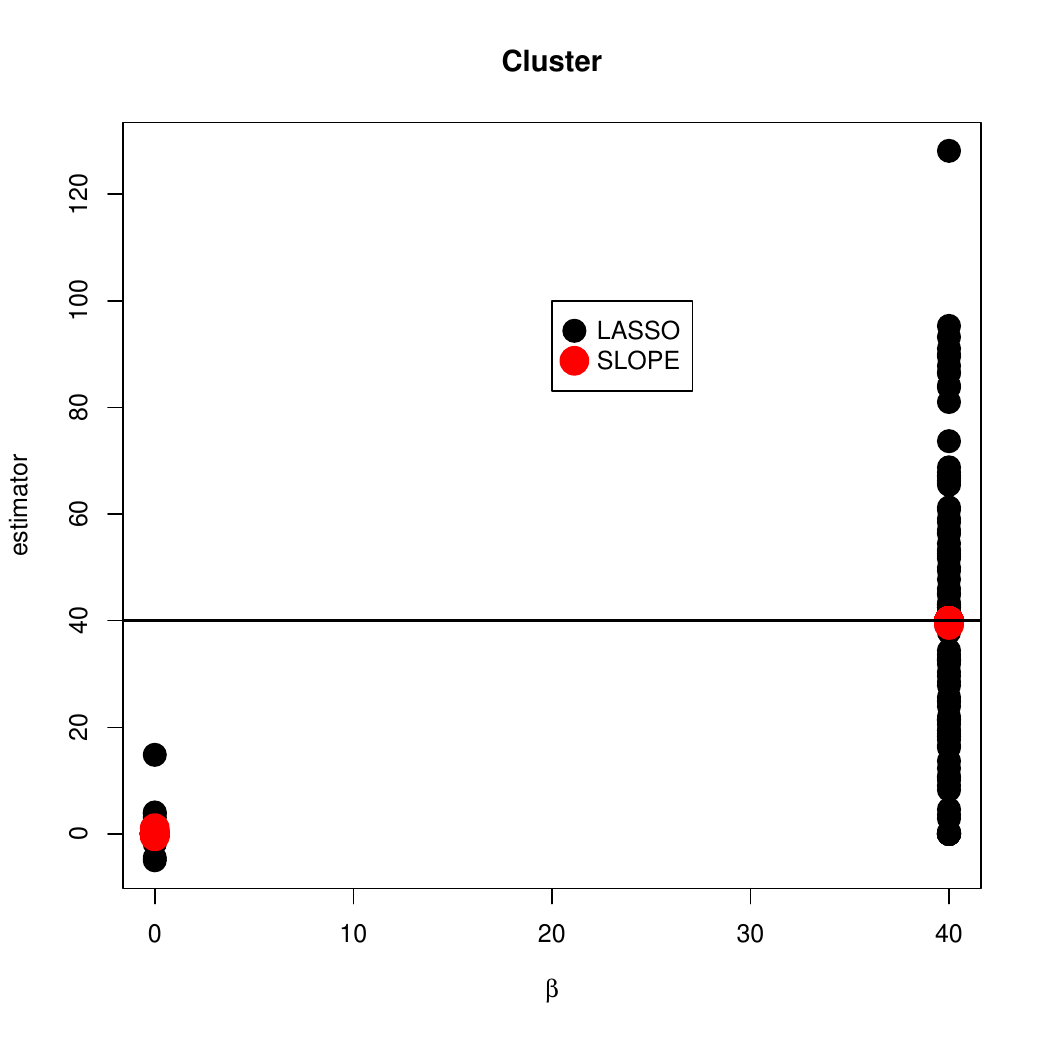}
\caption{Comparison of LASSO and SLOPE when the cluster structure is present in the data. Here $n=100$, $p=200$, and the correlation between $i^{th}$ and $j^{th}$ column of $X$ is equal to $0.9048^{|i-j|}$.  First $k=100$ columns of $X$ are associated with $Y$ and their nonzero regression coefficients are all equal to $40$. The SLOPE and LASSO irrepresentability conditions are not satisfied, but SLOPE, contrary to LASSO, satisfies the  accessibility condition.}
\label{fig:discussion}
\end{figure}
Figure \ref{fig:discussion} illustrates this phenomenon and shows that the accessibility condition for SLOPE can be much less restrictive than the accessibility condition for LASSO. In this example the matrix $X$ and the vector $Y$ are generated as in example illustrated in Figure \ref{fig:introduction} and the only difference is that now first $k=100=n$ regression coefficients are all equal to $40$.  In this situation the accessibility condition for LASSO is not satisfied and LASSO cannot properly separate the null and nonzero regression coefficients. Also, despite the selection of the tuning parameter so as to minimize the squared estimation error, the precision of LASSO estimates is very poor. As far as SLOPE is concerned, the irrepresentability condition is not satisfied but the accessibility condition holds. Thus, while SLOPE cannot properly identify the pattern, it  
estimates $\beta$ with such a good precision that the difference between the estimated and the true pattern is hardly visible on the graph. These favorable ranking and estimation properties of the SLOPE method enable pattern recovery through appropriately selected thresholded versions of SLOPE. We also expect that the mathematical understanding of SLOPE irrepresentability condition presented in this article will lead to the development of efficient adaptive versions of SLOPE, with improved estimation and pattern recovery properties.

The results presented in this article pave the way for a full understanding of the SLOPE pattern recovery properties. We expect that our SLOPE irrepresentability condition will be a basic block for proving further results on the pattern recovery of SLOPE and adaptive SLOPE in the high-dimensional regime. We also look forward to  research on other statistical models and loss functions. One specific focus of interest is the graphical SLOPE  (see \cite{graphSLOPE, Riccobello2025}), which could be used for identification of colored graphical models \cite{colors}, with specific parameter sharing patterns in the  precision matrix.  Such repetitive patterns occur naturally in many situations, \textit{e.g.}, in the case of the autoregressive type of dependence between variables in the database or  when variables are influenced by the same structural factors. We believe that an efficient exploitation of these unknown patterns by SLOPE will lead to a great reduction of the number of parameters and improvement of the graphical models estimation properties. 

Finally, we would like to recall that an interest in identifying the parameter sharing patterns goes beyond classical parametric models and is prevalent also in the modern machine learning community. As mentioned in the introduction, the prominent example is provided by the Convolutional Neural Networks (CNN), where the ``parameter sharing'' has made it possible to dramatically improve computational and statistical efficiency. 
While the parameter sharing in CNN is driven entirely by the expert knowledge, regularization by SLOPE allows to identify and exploit patterns based on the data. In principle one can also use SLOPE in the Bayesian context and combine the information in the data with the imprecise prior knowledge on possible parameter sharing patterns  (see \cite{ABSLOPE} for the preliminary version of adaptive Bayesian SLOPE). 
It is expected that recent developments in efficient implementations of the SLOPE optimization algorithm (see, \textit{e.g.} \cite{MBzeSzwedami,jonas}) will soon allow for an integration of SLOPE regularization with the deep neural network architectures.

\appendix 

\section{Proofs} \label{sec:proofs}

\subsection{Proof of Proposition \ref{prop:affine}}
Note that if $M=0$, then the statement holds by~\eqref{eq:subdiff_norm}. Thus we may later assume that $M\neq 0$. To ease the notation, we write $\tilde{\Lambda}$ instead of $\tilde{\Lambda}_M$. The elements of $\tilde{\Lambda}$ are denoted by $\tilde{\Lambda}_l$, $l=1,\ldots,k$.
Let $k=\|M\|_\infty$. Before proving  Proposition \ref{prop:affine} note that, by assumption, there exists $s\in \R^{k+}$  such that 
$b=U_Ms$. Consequently, $|b|_{\downarrow}=U_{|M|_{\downarrow}}s$ 
and thus
\[J_\Lambda(b)=\lambda_1|b|_{(1)}+\ldots+\lambda_p|b|_{(p)}=\Lambda'U_{|M|_{\downarrow}}s= \tilde \Lambda's= s_1\tilde\Lambda_1+\ldots+s_k\tilde \Lambda_k.
\]
Moreover, with $p_l=|\{i\colon |M_i|\geq k+1-l\}|$, we have $\tilde\Lambda_l=	\lambda_{p_{l-1}+1}+\ldots+\lambda_{p_l}$, $l=1,\ldots,k$.

\begin{proof}[Proof of Proposition \ref{prop:affine}]\ \\
	First we prove the inclusion
	$\partial{J_\Lambda}(b)\subset  \left\{ v\in\R^{p}\colon J_\Lambda^*(v)\le 1 \mbox{ and } U_{M}' v =  \tilde \Lambda \right\}$. 
	Let $v\in \partial{J_\Lambda}(b)$. Since $J^*_{\Lambda}(v)\le 1$ (see \eqref{eq:subdiff_norm})
	then, by definition of the dual sorted $\ell_1$ norm,  for all $j\in \{1,2,\ldots,p\}$ we have
	$\sum_{i=1}^{j}|v|_{(i)}\le\sum_{i=1}^{j}\lambda_i$. 
	It remains to prove that $U_M'v=\tilde \Lambda$.
	For all $l\in \{1,\ldots,k\}$ we have the following inequality
	\begin{multline}
		\label{eq:main}
		\sum\limits_{i=1}^{l} [U'_M v]_i =\sum\limits_{i\colon |M_i|\ge k+1-l}\sign(M_i)v_i \le \sum\limits_{i\colon |M_i|\ge k+1-l}|v_i|\\ \le
	\sum\limits_{i=1}^{p_l}|v|_{(i)} 
		\le 
		\sum\limits_{i=1}^{p_l}\lambda_i 
		= \sum\limits_{i=1}^{l} \tilde\Lambda_i.
	\end{multline}
	Note that
	\begin{multline*}
		b'v=(U_Ms)'v= \sum_{i=1}^k s_i [U'_M v]_i = 
		\sum_{l=1}^{k-1} (s_l-s_{l+1})\sum_{i=1}^{l} [U'_M v]_i + s_k\sum_{i=1}^{k}[U'_M v]_i 
		\\
		\leq 
		\sum_{l=1}^{k-1} (s_l-s_{l+1})\sum_{i=1}^{l} \tilde{\Lambda}_i + s_k\sum_{i=1}^{k}\tilde{\Lambda}_i= \sum_{l=1}^{k} s_l\tilde{\Lambda}_l= J_{\Lambda}(b).
	\end{multline*}
	Moreover, since $v\in\partial{J_\Lambda}(b)$, we have $b'v = J_{\Lambda}(b)$ (see \eqref{eq:subdiff_norm}).
	Therefore
	\[
	\sum_{i=1}^l [U'_M v]_i = \sum_{i=1}^l \tilde\Lambda_i \quad\mbox{ for }\quad l=1,\ldots,k
	\]
	and thus the inequalities given in \eqref{eq:main} are the equalities. Thus			
	\[ [U'_M v]_l =\tilde\Lambda_l\quad\mbox{ for }\quad l=1,\ldots,k
	\]
	and hence that $U'_M v = \tilde\Lambda$.
	
	Now we prove the other inclusion,
	$\partial{J_\Lambda}(b)\supset \left\{ v\in\R^{p}\colon J_\Lambda^*(v)\le 1 \mbox{ and } U_{M}' v =  \tilde \Lambda_M \right\}$. Assume that $v\in \R^p$ satisfies $J_\Lambda^*(v)\le 1$ and $U'_M v=\tilde\Lambda$.  To prove that $v\in \partial{J_\Lambda}(b)$ it remains to establish  that $b'v = J_{\Lambda}(b)$ 
	(see \eqref{eq:subdiff_norm}). Since  $b=U_Ms$, we have
	\begin{equation*}
		b'v = (U_M s)'v = s' U'_M v = 
		s' \tilde\Lambda=J_\Lambda(b).
	\end{equation*}
\end{proof}

\subsection{Proof of Proposition \ref{prop:subdiff_char}}
\begin{lemma}
	\label{lem:small_lemma}
	Let $\Lambda\in \R^{p+}$ and $b\in \R^p$. If  $\Lambda\in \partial{J_\Lambda}(b)$ then $b_1\ge \ldots \ge b_p\ge 0$.
\end{lemma}
\begin{proof}
	Let us assume that  $b_i<0$ for some $i\in \{1,\ldots,p\}$. For 
 \[
 \check \pi=(\lambda_1, \ldots,  \lambda_{i-1}, -\lambda_i,
	\lambda_{i+1}, \ldots, \lambda_p)
 \]
 we have $J_\Lambda^*(\check \pi)\le 1$ and 
	one may deduce  that 
	$$\Lambda'b<\check\pi'b\le \max \{\pi'b\colon J_\Lambda^*(\pi)\le 1 \}=J_\Lambda(b).$$
	Consequently $\Lambda\notin \partial{J_\Lambda}(b)$ leading to a contradiction. Let us assume that $b_i<b_j$ for some $1\le i <j\le p$. Let us define 
	$\check \pi$, where $J_\Lambda^*(\check \pi)\le 1$,  as follows 
	\[
	\check \pi_k=\begin{cases}
		\lambda_k & \mbox{ if } k\neq i, k\neq j,\\
		\lambda_j & \mbox{ if } k= i,\\
		\lambda_i & \mbox{ if } k= j,
	\end{cases}\qquad k = 1,\ldots,p.
	\]
	Since $\lambda_i>\lambda_j$, by the rearrangement inequality we have 
	$\lambda_ib_i+\lambda_jb_j<\lambda_jb_i+\lambda_ib_j$.
	Thus, one may deduce the following inequality
	\[
	\Lambda'b<\check \pi'b\le \max \{\pi'b\colon \pi\in \R^p\,\, J_\Lambda^*(\pi)\le 1 \}=J_\Lambda(b).
	\]
	Consequently $\Lambda\notin \partial{J_\Lambda}(b)$ leading to a contradiction.
\end{proof}

Let  $\psi$ be   an orthogonal transformation defined by \[\psi\colon  b \in \R^p\mapsto (v_1b_{r(1)},\ldots,v_pb_{r(p)})\] where $v_1,\ldots,v_p\in \{-1,1\}$ and $r$ is a permutation on $\{1,\ldots,p\}$.
Before proving Proposition \ref{prop:subdiff_char}  let us recall that for any $a,b\in \R^p$ we have $J_\Lambda(b)= J_\Lambda(\psi(b))$, 
$J_\Lambda^*(b)= J_\Lambda^*(\psi(b))$ and $b'a=\psi(b)'\psi(a)$ implying thus 
$\partial{J_\Lambda}(\psi(b))=\psi(\partial{J_\Lambda}(b))$. 

\begin{proof}[Proof of Proposition \ref{prop:subdiff_char}]
	If $\model(a)=\model(b)$ then, according to Proposition \ref{prop:affine},
	$\partial{J_\Lambda}(a)=\partial{J_\Lambda}(b)$. 
	Let us set $M=\model(a)$ and $\tilde M=\model(b)$,
	it remains to prove that if 
	$\partial{J_\Lambda}(a)= \partial{J_\Lambda}(b)$ then $M= \tilde M$.  Since the subdifferential $\partial{J_\Lambda}(a)$ depends on $a$ only through its pattern, then by Proposition \ref{prop:affine} we have $\partial{J_\Lambda}(a)=\partial{J_\Lambda}(M)$ and similarly  $\partial{J_\Lambda}(b)=\partial{J_\Lambda}(\tilde{M})$.
	
	First let us assume that $M= |M|_{\downarrow}$ namely $M_1 \ge M_2 \ge \ldots \ge M_p \ge 0$. In this case, $M'\bLambda=J_{\bLambda}(M)$ and hence $\bLambda=(\lambda_1,\ldots,\lambda_p)' \in \partial{J_\Lambda}(M)$. 	Since $\partial{J_\Lambda}(M)=\partial{J_\Lambda}(\tilde M)$, it follows  from Lemma \ref{lem:small_lemma}  that $\tilde M_1 \ge  \ldots \ge \tilde M_p \ge 0$, because $\bLambda \in \partial{J_\Lambda}(\tilde M)$. To prove that $M=\tilde M$, first let us establish that $M_p=\tilde M_p=0$ or $M_p=\tilde M_p=1$. If $M_p=0$ and $\tilde M_p=1$ then, let us set  $\check \pi=(\lambda_1,\ldots,\lambda_{p-1},0)'$, where $J_\Lambda^*(\check \pi)\le 1$. 
	Because
	\[
	J_\Lambda(M)=\Lambda'M=\check \pi'M \mbox{ and }    J_\Lambda(\tilde M)=\Lambda'\tilde M>\check \pi' \tilde M
	\]
	we have $\check \pi \in \partial{J_\Lambda}(M)$ and 
	$\check \pi \notin \partial{J_\Lambda}(\tilde M)$ which provides a contradiction. We proceed analogously for $M_p=1$ and $\tilde M_p=0$. To complete the proof that $M=\tilde M$, let us establish that $M_i= M_{i+1}$ and   $\tilde M_i=\tilde M_{i+1}$
	or $M_i> M_{i+1}$ and   $\tilde M_i>\tilde M_{i+1}$.
	If $M_i= M_{i+1}$ and   $\tilde M_i>\tilde M_{i+1}$
	then, let us define 
	$\check \pi$, where $J_\Lambda^*(\check \pi)\le 1$, as follows 
	$$
	\check \pi_k=\begin{cases}
		\lambda_k & \mbox{ if } k\neq i, k\neq i+1,\\
		\lambda_{i+1} & \mbox{ if } k= i,\\
		\lambda_i & \mbox{ if } k= i+1,
	\end{cases}\qquad  k =1,\ldots,p.
	$$
	Since $\lambda_i M_i+ \lambda_{i+1}M_{i+1}=\lambda_{i+1}M_i+\lambda_iM_{i+1}$ and 
	$\lambda_i \tilde M_i+\lambda_{i+1}\tilde M_{i+1}>
	\lambda_{i+1} \tilde M_i+\lambda_{i}\tilde M_{i+1}$ then
	\[
	J_\Lambda(M)=\Lambda'M=\check \pi'M \mbox{ and }    J_\Lambda(\tilde M)=\Lambda'\tilde M>\check \pi' \tilde M.
	\]
	Consequently $\check \pi \in \partial{J_\Lambda}(M)$ and 
	$\check \pi \notin \partial{J_\Lambda}(\tilde M)$ which provides a contradiction. We proceed analogously for $M_i> M_{i+1}$ and   $\tilde M_i=\tilde M_{i+1}$. Finally, if $M\neq |M|_{\downarrow}$  then let us pick an orthogonal transformation $\psi$ as defined above for which
	$\psi(M)=|M|_{\downarrow}$. Since $\partial{J_\Lambda}(M)=\partial{J_\Lambda}(\tilde M)$ implies that
	$\partial{J_\Lambda}(\psi(M))=\partial{J_\Lambda}(\psi(\tilde M))$,  the first part of the proof establishes that $\psi(\tilde M)=\psi(M)$ and thus $M=\tilde M$.  
	
\end{proof}

Recall that $J_\Lambda^*(x)\leq 1$ if and only if
\begin{equation}\label{eq:IC_j}
	|x|_{(1)}+\ldots+ |x|_{(j)} \le \lambda_1+\ldots+ \lambda_j,
	\qquad\qquad j=1,\ldots, p.
\end{equation}
The following result follows from the proof of Proposition \ref{prop:affine}. 
\begin{proposition}\label{pro:simple}
	Assume $x\in\R^p$ satisfies $J_\Lambda^*(x)\leq 1$ and let $b\in\R^p$. Then, $x$ belongs to $\partial J_\Lambda(b)$ if and only if the following three conditions hold true:
	\begin{enumerate}
		\item  If $b_i\neq 0$, then $\sign(x_i)=\sign(b_i)$,  
		\item If $|b_i| > |b_j|$ then $|x_i| \geq |x_j|$,
		\item The equalities hold in \eqref{eq:IC_j} for $j\in\{ n_1, n_2, \ldots, n_k\}$,
		where $n_j=|\{i\colon |M_i|\geq k+1-j\}|$ with $(M_1,\ldots,M_p)'=\model(b)$.
	\end{enumerate}
\end{proposition}

\subsection{Proof of Theorem \ref{th:PT+}}

\begin{proof}[Proof of Theorem~\ref{th:PT+}]
	{\it Necessity.} Let us assume  that there exists $\hat \beta\in S_{X,\Lambda}(Y)$ with $\model(\hat \beta)=M$. Consequently, $\hat \beta=\U_Ms$ for some $s\in \R^{k+}$.
	
	By Proposition \ref{prop:subdiff_char},
	$\X'(\Y-\X\hat \beta)\in \partial{J_\bLambda}(\hat{\beta})=\partial{J_\bLambda}(M)$.
	Multiplying this inclusion by $\U_M'$, 
	due to \eqref{eq:aff}, we get $\tilde \X_M'(\Y-\X\hat \beta)=\tilde \bLambda_M$  and so
	\begin{align}\label{eq:lastE}
		\tilde \X_M'\Y- \tilde \bLambda_M = \tilde \X_M'\X\hat \beta =\tilde \X_M'\tilde\X_M s.
	\end{align}
	The positivity condition is proven.  
	
	We apply $(\tilde \X_M')^+$ from the left to \eqref{eq:lastE}
	and use the fact that
	$\tilde{\P}_M=(\tilde \X_M')^+\tilde \X_M'$ is the projection onto $\col(\tilde \X_M)$. Since $\X\hat \beta\in \col(\tilde \X_M)$, we have $\tilde{\P}_M\X\hat \beta=\X\hat \beta$. 
	Thus,
	\[
	\tilde{\P}_M \Y- (\tilde \X_M')^+  \tilde \bLambda_M = \X\hat\beta.
	\]
	The above equality  gives the  subdifferential condition:
	\begin{eqnarray}\label{eq:pipi}
		\partial{J_\bLambda}(M) \ni \X'(\Y-\X\hat \beta)&=&	\X'(\Y-(\tilde{\P}_M\Y-  (\tilde \X_M')^+\tilde \bLambda_M))\\
		\nonumber	&=&\X'(\tilde \X_M')^+\tilde\bLambda_M+\X'(\I_n-\tilde{\P}_M)Y
		= \pi.
	\end{eqnarray}
	{\it Sufficiency.} Assume that  the positivity condition and the subdifferential conditions hold true.
	Then, by the positivity condition, one may pick 
	$s\in \R^{k+}$ for which 
	\begin{equation}
		\label{eq:cone}
		\tilde \bLambda_M = 	\tilde \X_M'\Y- \tilde \X_M'\tilde \X_M s.
	\end{equation} 
	Let us show that $\U_M s\in S_{X,\Lambda}(Y)$. By definition of $\U_M$, we have $\model(\U_Ms)=M$ thus $\partial{J_\Lambda}(U_Ms)=\partial{J_\Lambda}(M)$. Moreover,  using  \eqref{eq:pipi} and \eqref{eq:cone} one may deduce
	\begin{eqnarray*}
		\partial{J_\Lambda}(U_Ms) \ni\pi&=& \X'(\Y-\tilde{\P}_M\Y+  (\tilde \X_M')^+\tilde \bLambda_M)\\
		&=& \X'(\Y-\tilde{\P}_M\Y + (\tilde \X_M')^+(\tilde X_MY-\tilde X_M'\tilde X_Ms))\\
		&=& X'(Y-XU_Ms).
	\end{eqnarray*}	
	Consequently $\U_M s\in S_{X,\Lambda}(Y)$.
\end{proof}

\subsection{Proof of Corollaries \ref{cor:mainmr0}
and \ref{cor:mainmr0bis}}

\begin{proof}[Proof of Corollary \ref{cor:mainmr0}]
We will prove the implications $i) \Rightarrow iii) \Rightarrow iv) \Rightarrow ii) \Rightarrow i)$.\\\\
$i)\Rightarrow iii):$  Suppose there exist $\Lambda \in \R^{p+}$ and $\hat \beta\in S_{X,\Lambda}(X\beta)$ such that $\model(\hat \beta)=\model(\beta)$.   Then, by Theorem \ref{th:PT+} and since $\varepsilon=0$, 
	the subdifferential condition reads as:
	$X'(\tilde{X}_M')^{+}\tilde \Lambda_M\in \partial{J_\Lambda}(M)$.\\\\
$iii)\Rightarrow iv):$ The condition $X'(\tilde{X}_M')^{+}\tilde \Lambda_M\in \partial{J_\Lambda}(M)$
remains true when $\Lambda\in \R^{p+}$ is scaled by a scalar parameter $\alpha>0$. Indeed 
 $$X'(\tilde{X}_M')^{+}\tilde{(\alpha \Lambda)}_M= \alpha X'(\tilde{X}_M')^{+}\tilde{\Lambda}_M \in \alpha \partial{J_\Lambda}(M)=\partial{J_{\alpha \Lambda}}(M)$$
 Therefore, up to scaling of $\Lambda$,  for any $\lambda_1^0>0$ 
 there exists $\Lambda \in \R^{p+}$ with $\lambda_1<\lambda_1^0$, such that $X'(\tilde{X}_M')^{+}\tilde \Lambda_M\in \partial{J_\Lambda}(M)$.\\\\
 $iv)\Rightarrow ii):$ 
 To prove that SLOPE can recover the pattern of $\beta$ in the noiseless case, it remains to show that the positivity condition holds.  Since $\beta=U_Ms$
	for some $s\in \R^{k+}$, where $k=\|M\|_\infty$, and $Y=X\beta$, we have 
	$$\tilde X_M'Y- \tilde \Lambda_M=\tilde X_M'\tilde X_Ms -  \tilde \Lambda_M.$$
	Therefore, for $\lambda_1$  sufficiently small, we have $\tilde X_M'Y- \tilde \Lambda_M\in \tilde X_M'\tilde X_M\R^{k+}$, which proves the positivity condition. \\\\
$ii)\Rightarrow i):$ This implication follows directly by construction. 
 \end{proof}
 
 \begin{proof}[Proof of Corollary \ref{cor:mainmr0bis}]
 The proof of Corollary \ref{cor:mainmr0bis} follows by an analogous argument.
 \end{proof}
\subsection{Proof of Theorem \ref{th:bound}}

\begin{lemma}
	\label{lemma:affine_space}
	Let $0\neq b\in \R^p$ and $M=\model(b)$. Then the smallest affine space containing $\partial{J_\Lambda}(b)$ is $\aff(\partial{J_\Lambda}(b))=\{v\in \R^p\colon U_M'v=\tilde \Lambda_M\}$. 
\end{lemma}
\begin{proof}
	According to Proposition \ref{prop:affine} we have 
	\[\aff(\partial{J_\Lambda}(b))\subset \{v\in \R^p\colon U_M'v=\tilde \Lambda_M\}.\]
	Moreover, according to Theorem 4 in \citep{schneider2022geometry} we have 
	$$\dim(\aff(\partial{J_\Lambda}(b)))=\|M\|_\infty=
	\dim(\{v\in \R^p\colon U_M'v=\tilde \Lambda_M\}), $$
	which achieves the proof. 
\end{proof}

\begin{proof}[Proof of Theorem \ref{th:bound}]
		{\it (i) Sharpness of the upper bound.} According to Theorem \ref{th:PT+},  pattern recovery by SLOPE
		is equivalent to have simultaneously the positivity condition and the subdifferential condition satisfied. The upper bound \eqref{eq:upper_bound} coincides with the probability of the subdifferential condition. Thus to prove that this upper bound is sharp, it remains to show that the probability of the  positivity condition tends to $1$ when $r$ tends to $\infty$. Clearly the upper bound is reached when $\tilde\Lambda_M\notin\col(\tilde{X}_M')$ thus we assume hereafter that $\tilde\Lambda_M\in\col(\tilde{X}_M')$.
		Recall that $\beta^{(r)}=U_{M}s^{(r)}$ for $s^{(r)}\in\R^{k+}$ and thus $\tilde X_M'Y^{(r)}= \tilde X_M' \tilde{X}_M s^{(r)}+\tilde X'_M\beps$. 
		As $\tilde X'_M(\tilde X'_M)^+=\tilde X'_M\tilde{X}_M(\tilde{X}'_M \tilde{X}_M)^+$ is the projection on $\col(\tilde X_M')$,
		we obtain
		\begin{align*}
			\tilde{X}'_M Y^{(r)}-\alpha_r \tilde\Lambda_M &= \tilde{X}'_M \tilde{X}_M s^{(r)} - \alpha_r \tilde\Lambda_M + \tilde{X}'_M \beps\\
			&= \tilde{X}'_M \tilde{X}_M s^{(r)} -\alpha_r \tilde{X}'_M \tilde{X}_M (\tilde{X}'_M \tilde{X}_M)^{+}  \tilde\Lambda_M + \tilde{X}'_M \tilde{X}_M (\tilde{X}'_M \tilde{X}_M)^{+} \tilde{X}'_M \beps\\
			&= \tilde{X}'_M \tilde{X}_M \Delta_r \left(\frac{1}{\Delta_r}s^{(r)} - \frac{\alpha_r}{\Delta_r}(\tilde{X}'_M \tilde{X}_M)^{+}\tilde\Lambda_M + \frac{1}{\Delta_r}(\tilde{X}'_M \tilde{X}_M)^{+}\tilde{X}'_M \beps\right).
		\end{align*}
		Note that by the assumption on $\Delta_r$:
		\begin{itemize}
			\item the vector $s^{(r)}/\Delta_r\in \R^{k+}$ is (component-wise) larger than or equal to $(k,\ldots,1)$;
			\item $\lim_{r\to \infty}\alpha_r/\Delta_r=0$ and 
			$\lim_{r\to \infty}1/\Delta_r=0$.
		\end{itemize}
		Consequently, for $r$ large enough
		we have 
		$$\tilde{X}'_M Y^{(r)}-\alpha_r \tilde\Lambda_M\in \tilde X'_M\tilde X_M\R^{k+}.$$
		Since this fact is true for any realization of $\beps$, one may deduce that 
		$$\lim_{r \to \infty}\mathbb P\left(\tilde{X}'_M Y^{(r)}-\alpha_r \tilde\Lambda_M \in  \tilde{X}'_M \tilde{X}_M \R^{k+}\right)=1.$$	
		{\it (ii)	Pattern consistency.} In the proof of  the previous part, we see that
		positivity condition occurs when $r$ is sufficiently large.
		Thus it remains to prove that 
		subdifferential condition occurs as $r\to\infty$
		when  $X'(\tilde X_M')^+\tilde \Lambda_M\in \mathrm{ri}(\partial{J_\Lambda}(M))$. 
		First we observe that
		\begin{equation}
			\label{eq:asympt}
			X'(\tilde X_M')^+\tilde \Lambda_M +\frac{1}{ \alpha_r}X'(I_n-\tilde{P}_M)\varepsilon \overset{r \to \infty}{\longrightarrow}X'(\tilde X_M')^+\tilde \Lambda_M. 
		\end{equation}
		Note by Lemma \ref{lemma:affine_space} that $X'(\tilde X_M')^+\tilde \Lambda_M +\alpha_r^{-1}X'(I_n-\tilde{P}_M) 
		\varepsilon \in \mathrm{aff}(\partial{J_\Lambda}(M))$. Indeed, since $\tilde \Lambda_M\in \col(\tilde X_M')$ we have  
		$$\underbrace{U_M'X'(\tilde X_M')^+\tilde \Lambda_M}_{=\tilde \Lambda_M} +\frac{1}{\alpha_r}\underbrace{U_M'X'(I_n-\tilde{P}_M) 
			\varepsilon(\omega)}_{=0}=\tilde \Lambda_M.$$
		The second term above is zero due to the fact that $(I_n-\tilde{P}_M)$ is an orthogonal projection onto $\col(\tilde{X}_M)^\bot$.
		When $X'(\tilde X_M')^+\tilde \Lambda_M\in \mathrm{ri}(\partial{J_\Lambda}(M))$, due to \eqref{eq:asympt}, one may deduce that for sufficiently large $r$ we have
		\[
		X'(\tilde X_M')^+\tilde \Lambda_M +\frac{1}{\alpha_r}X'(I_n-\tilde{P}_M) \varepsilon\in \partial{J_\Lambda}(M).
		\]
		Consequently, when $r$ is sufficiently large, both the positivity and the subdifferential conditions occur  which, by Theorem \ref{th:PT+}, 
		concludes the proof. 
	\end{proof}
	
	\subsection{Proofs from Section \ref{sec:bign}}
	In this part we give proofs of Theorem \ref{th:bound_n} and Theorem \ref{th:asymp sufficient}.
	They are preceded by a series of simple lemmas. 
	For reader's convenience we recall the setting of Section \ref{sec:bign}.
	\begin{enumerate}
		\item[A.] $\beps_n=(\epsilon_1,\ldots,\epsilon_n)'$, where $(\epsilon_i)_i$ are i.i.d. centered with finite variance $\sigma^2$.
		\item[B1.] $n^{-1}\X_n' \X_n \stackrel{\mathbb{P}}{\longrightarrow} \C>0$.
		\item[B2.] $\frac{\max_{i=1,\ldots,n} |X_{ij}^{(n)}|}{\sqrt{\sum_{i=1}^n (X_{ij}^{(n)})^2}} \stackrel{\mathbb{P}}{\longrightarrow}0$, where $\X_n=(X_{ij}^{(n)})_{ij}$,
		for each $j=1,\ldots,p.$
		\item[B'.] Rows of $\X_n$ are i.i.d. distributed as $\xi$, where $\xi$ is a random vector whose components are linearly independent a.s. and such that $\E[\xi_i^2]<\infty$ for $i=1,\ldots,p$.
		\item[C.] $(\X_n)_n$ and $(\epsilon_n)_n$ are independent. 
	\end{enumerate}
	We consider a sequence of tuning parameters $(\Lambda_n)_n$ defined by $\bLambda_n=\alpha_n\bLambda$, where $\bLambda\in\R^{p+}$ is fixed and $(\alpha_n)_n$ is a sequence of positive numbers.
	
	To ease the notation, we write the clustered matrices and clustered parameters without the subscript indicating the model $M$, \textit{i.e.}, $\tilde{\bLambda} = \U'_{|M|\downarrow}\bLambda$, $\tilde{\bLambda}_n = \alpha_n \tilde{\bLambda}$ and $\tilde{\X}_n = \X_n \U_M$.
	
	\begin{lemma}\label{lem:4}
		\begin{enumerate}
			\item[(i)] Under A, B1, B2 and C,  
			\begin{align}
				\frac1{\sqrt{n}} \X_n' \beps_n & \stackrel{d}{\longrightarrow}
				Z\sim\mathrm{N}(0,\sigma^2 C). \label{eq:d} 
			\end{align}
			\item[(ii)] Under A, B1 and C, 
			\begin{align}
				\frac1n \X_n' \beps_n &\stackrel{\mathbb{P}}{\longrightarrow} 0.\label{eq:p}
			\end{align}
			\item[(iii)] Under A, B' and C,
			\begin{gather}
				0<\limsup_{n\to\infty} 	\frac{ \|X_n'\beps_n\|_\infty }{\sqrt{n\log\log n}}<\infty\qquad\mbox{a.s.} \label{eq:lsup} 
			\end{gather}
		\end{enumerate}
	\end{lemma}
	\begin{proof}
		Proof of \eqref{eq:d}. It is enough to show that for any Borel subset $A\subset \R^p$ one has
		\begin{align}\label{eq:d1}
			\mathbb{P}\left(\frac{1}{\sqrt{n}} \X_n' \beps_n \in A\mid (X_n)_n\right) \stackrel{\mathbb{P}}{\longrightarrow} \mathbb{P}\left(Z\in A\right).
		\end{align}
		Since both sides above are bounded, the convergence in probability implies convergence in $L^1$ and therefore establishes \eqref{eq:d}. To show \eqref{eq:d1} we will prove that for any subsequence $(n_k)_k$, there exists  a sub-subsequence $(n_{k_l})_l$ for which, as $l\to\infty$,
		\begin{align}\label{eq:d2}
			\mathbb{P}\left(\frac{1}{\sqrt{n_{k_l}}} \X_{n_{k_l}}' \beps_{n_{k_l}} \in A\mid (X_n)_n\right) \stackrel{a.s.}{\longrightarrow} \mathbb{P}\left(Z\in A\right).
		\end{align}
		Let $\mathbb{P}_{\mathbf{X}}$ denote the regular conditional probability $\mathbb{P}(\cdot\mid (X_n)_n)$ on $(\Omega,\mathcal{F})$. 
		By assumptions B1 and B2, from sequences $(n_k)_k$ one can choose a subsequence $(n_{k_l})_l$ for which 
		\[
		\frac{1}{n_{k_l}}\X_{n_{k_l}}' \X_{n_{k_l}} \stackrel{a.s.}{\longrightarrow} \C>0\quad\mbox{and}\quad  \frac{\max_{i=1,\ldots,{n_{k_l}}} |X_{ij}^{({n_{k_l}})}|}{\sqrt{\sum_{i=1}^{n_{k_l}} (X_{ij}^{({n_{k_l}})})^2}} \stackrel{a.s}{\longrightarrow}0,\quad
		j=1,\ldots,p.
		\]
		
		We have 
		\begin{align*}
			\mathrm{Var}_{\mathbf{X}}\left(\frac{1}{\sqrt{n_{k_l}}}X'_{n_{k_l}} \beps_{n_{k_l}}\right)&=\frac{1}{{n_{k_l}}}\E\left[X_{n_{k_l}}' \beps_{n_{k_l}} \beps_{n_{k_l}}'X_{n_{k_l}} \mid(X_{n})_{n} \right]\\ &=\frac{1}{{n_{k_l}}}X_{n_{k_l}}' \E\left[\beps_{n_{k_l}} \beps_{n_{k_l}}' \right]X_{n_{k_l}}  
			=\frac{\sigma^2 }{{n_{k_l}}}X_{n_{k_l}}' X_{n_{k_l}} \stackrel{a.s.}{\longrightarrow} \sigma^2 C>0,
		\end{align*}
		and one can apply multivariate Lindeberg-Feller CLT on the space $(\Omega,\mathcal{F},\mathbb{P}_{\mathbf{X}})$ to prove \eqref{eq:d2}. 
		Alternatively, the same result follows from \cite[Corollary 1.1]{CLT}\footnote{For our application, the assumption of nonnegative weights in \cite[Corollary 1.1]{CLT} is not essential.}, which concerns more general Central Limit Theorem for linearly negative quadrant dependent variables with weights forming a triangular array (in particular assumption B2 coincides with \cite[(1.8)]{CLT}). 
		
		For (ii) we observe that previous derivations imply that $\mathrm{Var}_{\mathbf{X}}(n^{-1}\X_n'\beps_n)\stackrel{\mathbb{P}}{\longrightarrow}0$. 
		We deduce that $\mathbb{P}_\textbf{X}(n^{-1}\|\X_n'\beps_n\|>\delta)\stackrel{\mathbb{P}}{\longrightarrow}0$ and hence (ii)
		follows
		upon averaging over $(\X_n)_n$.
		
		Eq. \eqref{eq:lsup} is the law of iterated logarithm for an i.i.d. sequence $(\xi_i \epsilon_i)_i$.
	\end{proof}
	
	\begin{lemma} \label{lemma:s_n} 
		Let $M=\model(\beta)$. Assume $\alpha_n/n\to 0$.
		\begin{enumerate}
			\item[(i)] Under A, B1 and C, the positivity condition is satisfied for large $n$ with high probability. 
			\item[(ii)] Under A, B' and C, the positivity condition is almost surely satisfied for large $n$.
		\end{enumerate}
	\end{lemma}
	\begin{proof}
		If $M=0$, then the positivity condition is trivially satisfied. Thus, we consider $M\neq0$. 
		
		(i) Since $\tilde{\X}_n' \tilde{\X}_n$ is invertible for large $n$ with high probability, the positivity condition is equivalent to
		\[
		s_n:= 
		(\tilde{\X}_n' \tilde{\X}_n)^{-1}[\tilde{\X}_n' \Y_n - \tilde{\bLambda}_n] \in \R^{k+}.
		\]
		Let $s_0\in\R^{k+}$ be defined through $\beta = \U_M s_0$, where $k=\|M\|_\infty$.
		We will show that if $\alpha_n/n\to 0$, then 
		$s_n  \stackrel{\mathbb{P}}{\longrightarrow} s_0$.
		Since $\R^{k+}$ is an open set, this will imply that for large $n$ with high probability, the positivity condition is satisfied. 
		
		First we rewrite $s_n$ as 
		\[
		s_n= 
		(\tilde{\X}_n' \tilde{\X}_n)^{-1}\tilde{\X}_n' \Y_n - \alpha_n(\tilde{\X}_n' \tilde{\X}_n)^{-1}\tilde{\bLambda}.
		\]
		Since $\beta=\U_M s_0$, we conclude $\X_n\beta=\X_n \U_M s_0 = \tilde{\X}_n s_0$, so the linear regression model takes the form 
		$\Y_n = \tilde{\X}_n s_0+\beps_n$. Thus,  $(\tilde{\X}_n'\tilde{\X}_n)^{-1}\tilde{\X}_n' \Y_n$ is the OLS estimator of $s_0$. 
		
		By assumption B and Lemma \ref{lem:4}, we deduce that 
		\[
		(\tilde{\X}_n'\tilde{\X}_n)^{-1}\tilde{\X}_n' \Y_n =s_0 +  (n^{-1}\tilde{\X}_n'\tilde{\X}_n)^{-1} U_M \frac{1}{n}X_n' \beps_n \stackrel{\mathbb{P}}{\longrightarrow} s_0+ [(\U_M' \C\U_M)^{-1} U_M] 0=s_0.
		\] 
		To complete the proof, we note that 
		\[
		\alpha_n(\tilde{\X}_n' \tilde{\X}_n)^{-1}\tilde{\bLambda} = \frac{\alpha_n}{n}\left[ n(\tilde{\X}_n'\tilde{\X}_n)^{-1}\tilde{\bLambda}\right]\stackrel{\mathbb{P}}{\longrightarrow} 0\, \left[(\U_M' \C\U_M)^{-1} \tilde{\bLambda}\right] = 0.
		\]
		(ii) If one assumes B' instead of B1, then $n^{-1}\X_n'\X_n \stackrel{a.s.}{\longrightarrow} C$ and by \eqref{eq:lsup}, $n^{-1} \X_n' \beps_n \stackrel{a.s.}{\longrightarrow} 0$. The result follows along the same lines as (i).
	\end{proof}
	
	For $M\neq 0$ we denote 
	\begin{gather*}
		\pi_n^{(1)} = \X_n' (\tilde{\X}'_n)^+  \tilde{\bLambda}_n, \qquad\qquad 
		\pi_n^{(2)} =\X'_n (\I_n-\tilde{\P}_n) \Y_n,\\
		\pi_n = \pi_n^{(1)}+\pi_n^{(2)},
	\end{gather*}
	which simplifies in the $M=0$ case to $\pi_n = \pi_n^{(2)} =  X_n' Y_n$.
	
	Recall that the subdifferential condition is equivalent to  $J^*_{\bLambda_n}(\pi_n)\leq 1$ and $\tilde{\Lambda}_n\in\col(\tilde{X}'_M)$ and the latter is satisfied in our setting. Since $\alpha J_{\Lambda}=J_{\alpha\Lambda}$, the subdifferential condition is satisfied if and only if 
	\[
	1\geq J^*_{\bLambda}\left(\alpha_n^{-1}\pi_n\right)=J^*_{\bLambda}\left( \alpha_n^{-1}\pi_n^{(1)} + \frac{\sqrt{n}}{\alpha_n } n^{-1/2} \pi_n^{(2)}\right).
	\]
	In view of results shown below, $\alpha_n^{-1}\pi_n^{(1)}$ converges almost surely, while $n^{-1/2} \pi_n^{(2)}$ converges in distribution to a Gaussian vector. Thus, the pattern recovery properties of SLOPE estimator strongly depend on the behavior of the sequence $(\alpha_n/\sqrt{n})_n$. 
	
	\begin{lemma}\label{lem:DN0}
		(a)
		\begin{enumerate}
			\item[(i)]  Assume A, B1 and C. If $M\neq0$, then
			\[
			\frac{1}{\alpha_n}\pi_n^{(1)} \stackrel{\mathbb{P}}{\longrightarrow}    \C\U_M(\U_M'\C\U_M)^{-1}\tilde \bLambda.
			\]
			\item[(ii)] Assume A, B1, B2 and C.
			The sequence $\left(n^{-1/2}\pi_n^{(2)}\right)_n$ converges in distribution to a Gaussian vector $Z$ with
			\[
			\Zb\sim\mathrm{N}\left(0, \sigma^2\left[\C-\C\U_M(\U_M'\C\U_M)^{-1}\U_M'\C\right]\right).
			\]
			\item[(iii)] Assume A, B1 and C. If 
			$\lim_{n \to \infty}  {\alpha_n}/{\sqrt{n}} = \infty$,
			then $\alpha_n^{-1}\pi_n^{(2)}\stackrel{\mathbb{P}}{\longrightarrow}0$.
		\end{enumerate}
		(b) Assume A, B' and C. 
		\begin{enumerate}
			\item[(i')] If $M\neq0$, then 
			\[
			\frac{1}{\alpha_n}\pi_n^{(1)} \stackrel{a.s.}{\longrightarrow}    \C\U_M(\U_M'\C\U_M)^{-1}\tilde \bLambda.
			\]
			\item[(ii')] If
			$\lim_{n \to \infty}  {\alpha_n}/{\sqrt{n\log\log n}} = \infty$,
			then $\alpha_n^{-1}\pi_n^{(2)}\stackrel{a.s.}{\longrightarrow}0$.
		\end{enumerate}
	\end{lemma}
	\begin{proof}
		\begin{enumerate}
			\item[(i)] Assumption B1 implies that 
			\[
			\X_n' \tilde{\X}_n (\tilde{\X}'_n \tilde{\X}_n)^{-1} = \frac1n \X_n' \X_n \U_M(\U_M'n^{-1}\X'_n \X_n\U_M)^{-1} \stackrel{\mathbb{P}}{\longrightarrow}\C\U_M (\U_M' \C\U_M)^{-1}.
			\]
			\item[(ii)] When $\beta=\U_M s_0$, then the linear regression model takes the form  $\Y_n = \tilde{\X}_n s_0+\beps_n$. Since $\tilde{\P}_n$ is the projection matrix onto  $\col(\tilde{\X}_n)$,  we have $(\I_n-\tilde{\P}_n)\tilde{\X}_n=0$. Thus,  
			\begin{align*}
				n^{-1/2}\pi_n^{(2)} &= n^{-1/2} \X_n' (\I_n-\tilde{\P}_n)\Y_n = n^{-1/2} \X_n' (\I_n-\tilde{\P}_n)\beps_n \\
				& = \left[\I_p-\X_n'\X_n \U_M(\U_M'\X_n'\X_n\U_M)^{-1}\U_M'\right] \left[n^{-1/2}\X_n'\beps\right].
			\end{align*}
			By assumption B1 we have,
			\begin{align}\label{eq:ddd}
			n^{-1}\X_n'\X_n \U_M(\U_M'n^{-1}\X_n'\X_n\U_M)^{-1}\U_M'\stackrel{\mathbb{P}}{\longrightarrow}\C\U_M  (\U_M' \C\U_M)^{-1} \U_M'.
			\end{align}
			Thus, by Lemma \ref{lem:4} (i) and Slutsky's theorem, we obtain (ii). 
			(iii) follows similarly as \ref{lem:4} (ii): with the aid of \eqref{eq:ddd} we show that \mbox{$\mathrm{Var}_{\textbf{X}}(\alpha_n^{-1} \pi_n^{(2)})\stackrel{\mathbb{P}}{\longrightarrow}0$}, which implies that conditionally on $(X_n)_n$ we have $\alpha_n^{-1} \pi_n^{(2)}\stackrel{\mathbb{P}_{\mathbf{X}}}{\longrightarrow}0$. 
			
			Assumption B' implies that $n^{-1}\X_n'\X_n\stackrel{a.s.}{\longrightarrow}C$ and thus (i') is proven in the same way as (i). (ii') follows from \eqref{eq:lsup}.
			
		\end{enumerate}
	\end{proof}

	\begin{proof}[Proof of Theorem \ref{th:bound_n}]
		(i) is a direct consequence of Lemmas \ref{lemma:s_n} and  \ref{lem:DN0}.
		Since positivity condition is satisfied for large $n$ with high probability, for (ii) we have with $M=\model(\beta)$,
		\begin{align}\label{eq:ech}
			\lim_{n\to\infty}\mathbb{P}\left(\model(\betaSLOPE_n)=M\right) & = \lim_{n\to\infty}\mathbb{P}\left( \pi_n \in \partial J_{\alpha_n\Lambda}(M)\right) 
			=  \lim_{n\to\infty}\mathbb{P}\left( \alpha_n^{-1}\pi_n \in \partial J_{\Lambda}(M)\right)\\
			&\geq \lim_{n\to\infty}\mathbb{P}\left( \alpha_n^{-1}\pi_n \in \mathrm{ri}(\partial J_{\Lambda}(M))\right)=1,\nonumber
		\end{align}
		where in the last equality we use the Portmanteau Theorem, assumption \eqref{eq:Cri} and the fact that  sequence $(\alpha_n^{-1}\pi_n)_n$ converges in distribution to $\C\U_M(\U_M'\C\U_M)^{-1}$ if and only if $\alpha_n/\sqrt{n}\to \infty$. 
		
		Condition \eqref{eq:matrixIR} implies that $\C\U_M(\U_M'\C\U_M)^{-1}\in\partial J_\Lambda(M)$. Since $(\alpha_n^{-1}\pi_n)_n$ converges in probability to $\C\U_M(\U_M'\C\U_M)^{-1}$, the necessity of this condition is explained by \eqref{eq:ech}.
	\end{proof}
	
	\begin{proof}[Proof of Theorem \ref{th:asymp sufficient}]
		By Lemma \ref{lemma:s_n}, the positivity condition is satisfied for large $n$ almost surely. 
		By Lemma \ref{lem:DN0} (i) and (iii), we have 
		\[
		a_n:=\frac{1}{\alpha_n}\pi_n\stackrel{a.s.}{\longrightarrow} \C\U_M(\U_M'\C\U_M)^{-1}\tilde{\bLambda}=:a_0.
		\]
		It is easy to see that $U_M'a_n=\tilde{\Lambda}$.
		By the condition $a_0\in \mathrm{ri}(J_\Lambda(M))$ it follows that $a_n \in J_\Lambda(M)$ almost surely for sufficiently large $n$. Therefore $\pi_n \in J_{\Lambda_n}(M)$ for large $n$ almost surely and thus the subdifferential condition is also satisfied.
	\end{proof}

	\section{Refined results on strong consistency of the SLOPE pattern}\label{appB}
	In this appendix we aim to give  weaker assumptions on the design matrix than condition B', but which  ensure the almost sure convergence of the  pattern of $\betaSLOPE_n$. 
	\begin{enumerate}
		\item[A'.] $\beps_n=(\epsilon_1,\ldots,\epsilon_n)'$, where $(\epsilon_i)_i$ are independent random variables such that 
		\begin{align}\label{eq:epsA}
			\E[\epsilon_n]=0\quad\mbox{and}\quad\mathrm{Var}(\epsilon_n)=\sigma^2\quad\mbox{for all $n$,}\quad\mbox{ and } \sup_n\E[|\epsilon_n|^r]<\infty
		\end{align}
		for some $r>2$.
		\item[B''.] A sequence of design matrices $\X_1, \X_2,\ldots$ satisfies the condition
		\begin{equation} \label{assumption on design2}
			\frac{1}{n}\X_n' \X_n \stackrel{a.s.}{\longrightarrow} \C,
		\end{equation}
		where $\C$ is a deterministic positive definite symmetric $p\times p$ matrix.
		
		With $\X_n=\left(X_{ij}^{(n)}\right)_{ij}$,
		\begin{align}\label{eq:badass0}
			\lim_{n\to\infty} \frac{(\log n)^\rho}{\sqrt{n}} \sup_{i,j} \left|X_{ij}^{(n)}\right| = 0\quad\mbox{a.s. for all $\rho>0$}
		\end{align}
		and
		there exist nonnegative random variables $(c_i)_i$, constants $d>2/r$ and $m_0\in\mathbb{N}$ such that for $n>m\geq m_0$,
		\begin{gather}\label{eq:badass}
			\sup_j\left[\sum_{i=1}^m \left(X_{ij}^{(n)}-X_{ij}^{(m)}\right)^2 + \sum_{i=m+1}^n \left(X_{ij}^{(n)}\right)^2\right]\leq \left(\sum_{i=m+1}^n c_i\right)^d\quad\mbox{ a.s.}, \\
			\left(\sum_{i=m_0}^n c_i\right)^d = O(n)\quad\mbox{a.s.} \label{eq:badass2}
		\end{gather}
		\item[C.] $(\X_n)_n$ and $(\epsilon_n)_n$ are independent. 
	\end{enumerate}
	
	We note that conditions \eqref{eq:badass0} and \eqref{eq:badass} are trivially satisfied in the i.i.d. rows setting of Remark \ref{rem:rem0} or assumption B'. 
	The main ingredient of the proof of the strong pattern consistency is the law of iterated logarithm \eqref{eq:lsup} which holds trivially under B'. Below, we establish the same result under more general B''. The technical assumption \eqref{eq:badass} is a kind of weak continuity assumption on the rows of $\X_n$ as it says that the $\ell_2$-distance between $j$th rows of $\X_n$ and $\X_m$ should not be too large. 
	
	\begin{lemma}\label{lem:new}
		Assume A', B'' and C. 
		Then
		\begin{align}\label{eq:LLN}
			\limsup_{n\to\infty} 	\frac{ \|X_n'\beps_n\|_\infty }{\sqrt{n\log\log n}}<\infty\qquad\mbox{a.s.}
		\end{align}
	\end{lemma}
	\begin{proof}
		In view of \eqref{assumption on design} we have for $j=1,\ldots,p$,
		\begin{align}\label{eq:Anj}
			n^{-1} A_n^{(j)} := n^{-1} \sum_{i=1}^n \left(X_{ij}^{(n)}\right)^2=\left(n^{-1}\X_n'\X_n\right)_{jj} \stackrel{a.s.}{\longrightarrow} C_{jj}>0.
		\end{align}
		We apply the general law of iterated logarithm for weights forming a triangular array from \citep{LW82}. 
		The result follows directly from \cite[Theorem 1]{LW82}. Defining $a_{ni}^{(j)}:=X_{ij}^{(n)}$ for $i=1,\ldots,n$, $j=1,\ldots,p$, $n\geq1$ and $0$ otherwise, we have 
		\[
		(\X_n'\beps_n)_j = \sum_{i=-\infty}^\infty a_{ni}^{(j)}\epsilon_i
		\]
		and therefore we fall within the framework of \cite[Eq. (1.3)]{LW82}. Then, \eqref{eq:epsA}, \eqref{eq:badass0}, \eqref{eq:badass} and \eqref{eq:badass2} coincide with \cite[(1.2), (1.6), (1.7), (1.8)]{LW82} respectively. Let $\mathbb{P}(\cdot|(\X_n)_n)$ be a regular conditional probability. Then, applying \cite[Theorem 1 (i)]{LW82} on the probability space $(\Omega,\mathcal{F},\mathbb{P}_{\mathbf{X}})$ to our sequence we obtain that for $j=1,\ldots,n$,
		\[
		\mathbb{P}\left( \limsup_{n\to\infty}  \frac{ \left|(\X_n'\beps_n)_j \right|}{\sqrt{2 A_n^{(j)}\log \log A_n^{(j)}}}\leq \sigma\Bigg| (\X_n)_n \right) =1\qquad \mbox{a.s.}
		\]
		Averaging over $(\X_n)_{n}$ and using \eqref{eq:Anj} again, we obtain the assertion.
	\end{proof}

	\begin{theorem}\label{th:asymp sufficient2}
		Assume A', B'' and C.  Suppose that $(\alpha_n)_n$ satisfies
		\begin{align*}
			\lim_{n \to \infty}  \frac{\alpha_n}{n}=0\qquad\mbox{and}\qquad		\lim_{n \to \infty}  \frac{\alpha_n}{\sqrt{n \log\log n} } = \infty.
		\end{align*}
		If  \eqref{eq:Cri} is satisfied, then
		$\model(\betaSLOPE_n)\stackrel{a.s.}{\longrightarrow}\model(\beta)$.
	\end{theorem}
	
	Comments:
	\begin{enumerate}
		\item[a)] Under reasonable assumptions (see \textit{e.g.} \cite[Theorem 1 (iii)]{LW82}) one can show that  
		\begin{align*}
			\limsup_{n\to\infty} 	\frac{ \|X_n'\beps_n\|_\infty }{\sqrt{n\log\log n}}>0\qquad\mbox{a.s.}
		\end{align*}
		Since $\alpha_n^{-1}\X_n'\beps_n\stackrel{a.s.}{\longrightarrow}0$ is necessary for the a.s. pattern recovery, we can show that the condition $\alpha_n/\sqrt{n\log\log n}\to\infty$ cannot be weakened. Thus, the gap between the convergence in probability and the a.s. convergence is integral to the problem and in general cannot be reduced.
		\item[b)] One can relax assumption B'' by imposing stronger conditions on the error $\beps_n$. \textit{E.g.} if $\beps_n$ is Gaussian, then one can use results from \citep{S84}. 
		We note that \citep{S84} offers a very similar result as \citep{LW82}, but their assumptions are not quite comparable, see \cite[Section 3 i)]{S84} for detailed discussion.
		\item[c)] For Gaussian errors, one can consider a more general setting where one does not assume any relation between $\beps_n$ and $\beps_{n+1}$, \textit{i.e.}, the error need not be incremental. For orthogonal design such approach was taken in \citep{Skalski_2022}. It is proved there that one obtains the a.s. SLOPE pattern consistency with the second limit condition of Theorem \ref{th:asymp sufficient2} replaced by $\lim_{n\to\infty} \alpha_n/\sqrt{n\log n}=\infty$. This result can be generalized to non-orthogonal designs.
	\end{enumerate}
	
	\section{Strong consistency of SLOPE estimator}

	\begin{lemma}\label{lem0}
		Assume that $\beps_n=(\epsilon_1,\ldots,\epsilon_n)'$ with $(\epsilon)_i$ i.i.d., centered and having finite variance. Suppose 
		\begin{align}\label{assumption on design3}
			\frac{1}{n}\X_n'\X_n \stackrel{a.s.}{\longrightarrow}C>0.
		\end{align}
		and that $(\beps_n)_n$ and $(\X_n)_n$ are independent.
		Then $n^{-1}\X_n'\beps_n\stackrel{a.s.}{\longrightarrow}0$.
	\end{lemma}
	\begin{proof}
		Let $\mathbb{P}(\cdot\mid (\X_n)_n)$ denote the regular conditional probability. 
		By \cite[Th. 1.1]{SLLN} applied to a sequence $(n^{-1}\X_n'\beps_n)_j$ on the probability space $(\Omega,\mathcal{F},\mathbb{P}(\cdot\mid (\X_n)_n))$,
		we obtain
		\[
		\mathbb{P}\left(\lim_{n\to\infty} n^{-1}(\X_n'\beps_n)_j =0\mid (\X_n)_n\right)=1,\qquad j=1,\ldots,p,\qquad\mbox{a.s.}
		\]
		Thus, applying the expectation to both sides above we obtain the assertion.
	\end{proof}
	
	\begin{theorem}\label{thm:consist}
		Assume that $\Y_n=\X_n\beta+\beps_n$, where $\beta\in\R^p$, $\beps_n=(\epsilon_1,\ldots,\epsilon_n)'$ with $(\epsilon)_i$ i.i.d., centered and finite variance. Suppose \eqref{assumption on design3} and that  $(\beps_n)_n$ and $(\X_n)_n$ are independent.
		Let $\bLambda_n=(\lambda_1^{(n)},\ldots,\lambda_p^{(n)})'$. 
		Then, for large $n$, $\Snew{\X_n}{\bLambda_n}{\Y_n}=\{\betaSLOPE_n\}$ almost surely.
		
		If $\beta\neq0$, then $\betaSLOPE_n\stackrel{a.s.}{\longrightarrow}\beta$ if and only if 
		\begin{align}\label{eq:alpha/nA}
			\lim_{n\to\infty}\frac{\lambda_1^{(n)}}{n}= 0.
		\end{align}
		If $\beta=0$ and \eqref{eq:alpha/nA} holds true, then $\betaSLOPE_n\stackrel{a.s.}{\longrightarrow}0$.
	\end{theorem}
	\begin{proof}[Proof of Theorem~\ref{thm:consist}]
		The assumption \eqref{assumption on design3} implies that the matrix $X_n'X_n$ is positive definite for large $n$ almost surely and hence ensuring that $\ker(\X_n)=\{0\}$. It is known that under trivial kernel, the set of SLOPE minimizers contains one element only.
		
		By Proposition \ref{prop:affine}, $\betaSLOPE_n$ is the SLOPE estimator of $\beta$ in a linear regression model $\Y_n=\X_n\beta+\beps_n$ if and only if for $\pi_n=\X_n'(\Y_n-\X_n\betaSLOPE_n)$ we have
		\begin{align}\label{eq:cond1}
			J_\Lambda^*(\pi_n)\leq 1
		\end{align}
		and
		\begin{align}\label{eq:cond2}
			\U_{M_n}' \pi_n = \tilde\bLambda_n,
		\end{align}
		where $M_n=\model(\betaSLOPE_n)$ and $\tilde\bLambda_n=U_{|M_n|\downarrow}'\bLambda_n$.
		By the definition of $\pi_n$ we have
		\[
		\betaSLOPE_n=(X_n' X_n)^{-1}X_n'  Y_n-(X_n'X_n)^{-1}\pi_n = \betaOLS_n-\left(\frac1nX_n'X_n\right)^{-1}\left(\frac1n\pi_n\right).
		\]
		Since in our setting $\betaOLS_n$ is strongly consistent, $\hat\beta^{SLOPE}_n \stackrel{a.s.}{\longrightarrow} \beta$ if and only if \\
		$(n^{-1}X_n'X_n)^{-1}\left(n^{-1}\pi_n\right)\stackrel{a.s.}{\longrightarrow}0$. 
		In view of \eqref{assumption on design3}, we have $(n^{-1}X_n'X_n)^{-1}\left(n^{-1}\pi_n\right)\stackrel{a.s.}{\longrightarrow}0$ if~and~only~if $n^{-1}\pi_n\stackrel{a.s.}{\longrightarrow}0$.
		
		Assume $n^{-1}\lambda_1^{(n)}\to0$. 
		By \eqref{eq:cond1} we have $\| \pi^n\|_\infty\leq \lambda_1^{(n)}$, which gives
		\[
		\left\| \frac{\pi_n}{n}\right\|_\infty \leq \frac{\lambda_1^{(n)}}{n}\to0.
		\]
		Therefore, \eqref{eq:alpha/nA} implies that $\hat\beta^{SLOPE}_n \stackrel{a.s.}{\longrightarrow} \beta$.
		
		Now assume that $\beta\neq0$ and $\betaSLOPE_n$ is strongly consistent, \textit{i.e.}, $n^{-1}\pi_n\stackrel{a.s.}{\longrightarrow}0$.
		Then, \eqref{eq:cond2} gives 
		\begin{align}\label{eq:ineq2cond}
			p \|\pi_n\|_\infty\geq\| U_{M_n}' \pi_n \|_\infty  = \| \tilde\Lambda_n\|_\infty \geq \lambda_1^{(n)}
		\end{align}
		provided $M_n\neq0$.
		Applying \eqref{eq:cond1} for $\hat\beta^{SLOPE}_n = 0$, we note that $M_n(\omega)=0$ if and only if 
		\begin{align*}
			J_{n^{-1}\Lambda_n}^*\left(n^{-1}X_n(\omega)' Y_n(\omega)\right)\leq1.
		\end{align*} 
		In view of Lemma \ref{lem0}, it can be easily verified that $n^{-1}X_n' Y_n\stackrel{a.s.}{\longrightarrow}C\beta$.
		Since 
		\[
		\left\|\frac{1}{n}\pi_n\right\|_\infty\quad\geq\quad \left\|\frac{1}{n}\pi_n \right\|_\infty {\bf 1}_{(M_n=0)}\quad=\quad\left\|\frac{1}nX_n' Y_n \right\|_\infty {\bf 1}_{(M_n=0)},
		\]
		we see that for $\beta\neq0$, we have $M_n\neq0$ for large $n$ almost surely. Thus, for $\beta\neq0$ we eventually obtain for large $n$
		\[
		\frac{\lambda_1^{(n)}}{n}\leq  p\left\| \frac{\pi_n}{n}\right\|_\infty\quad\mbox{a.s.}
		\]
	\end{proof}

	\section{\texorpdfstring{Geometric interpretation of $\X'(\tilde \X_M')^+\tilde \Lambda_M$}{Geometric interpretation of X'(tilde XM')+tilde LambdaM}}\label{sec:geom}
	Let $0\neq \beta\in \R^p$ where $\model(\beta)=M$. 
	For a SLOPE minimizer $\hat \beta\in S_{\X,\alpha\Lambda}(X\beta)$ the following  occurs:
	$$\frac{1}{\alpha}X'X(\beta-\hat \beta)\in \partial{J_\Lambda}(\hat \beta).$$
	In addition when $\model(\hat \beta)=M$, then the following facts hold:
	\begin{itemize}
		\item $\beta-\hat\beta\in \col(U_M)$, 
		so that $\frac{1}{\alpha}X'X(\beta-\hat \beta) \in \X'\X \,\col(U_M)$.
		\item $\partial{J_\Lambda}(\hat \beta)=\partial{J_\Lambda}(M)$.
	\end{itemize} Therefore, the noiseless pattern recovery  by SLOPE  clearly 
	implies that the
	vector space $X'X\col(U_M)=\col(X'\tilde X_M)$  intersects  $\partial{J_\Lambda}(M)$. Actually, the vector $\bar \Pi =X'(\tilde{X}'_M)^{+}\tilde{\Lambda}_M$ appearing in Corollary \ref{cor:mainmr0} has a geometric interpretation given in Proposition \ref{prop:geom}.
	\begin{proposition}\label{prop:geom}
		Let $X\in \R^{n\times p}$, $0\neq M\in \mathcal{P}^{\rm SLOPE}_p$ and $\Lambda \in \R^{p+}$. We recall that 
		$\tilde X_M=XU_M$,  $\tilde \Lambda_M=U_{|M|_{\downarrow}}'\Lambda$ and $\bar \Pi =X'(\tilde{X}'_M)^{+}\tilde{\Lambda}_M$. 
		We have the following statements:
		\begin{enumerate}
			\item[i)] If $\tilde{\Lambda}_M\notin\col(\tilde X_M')$ then $\aff(\partial{J_\Lambda}(M))\cap \col(X'\tilde X_M)=\emptyset$.
			\item[ii)] If $\tilde{\Lambda}_M\in\col(\tilde X_M')$ then $\aff(\partial{J_\Lambda}(M))\cap \col(X'\tilde X_M)=\{\bar \Pi\}$.
			\item[iii)] Pattern recovery by SLOPE in the noiseless case is equivalent to
			$\col(X'\tilde{X}_M)\cap \partial{J_\Lambda}(M)\not=\emptyset$.
		\end{enumerate}
	\end{proposition}
	\begin{proof}
		i) We recall that, according to Lemma \ref{lemma:affine_space}, $\aff(\partial{J_\Lambda}(M))=\{v\in \R^p\colon U_M'v=\tilde \Lambda_M \}$.
		If $\aff(\partial{J_\Lambda}(M))\cap \col(X'\tilde X_M)\neq \emptyset$ then  there exists $z\in \R^k$, where $k=\|M\|_\infty$, such that $X'\tilde X_Mz\in \aff(\partial{J_\Lambda}(M))$.
		Consequently, $\tilde \Lambda_M=U_M'X'\tilde X_Mz=\tilde X_M'\tilde X_M'z$  thus $\tilde{\Lambda}_M\in\col(\tilde X_M')$ which establishes i). \\
		ii) If $\tilde{\Lambda}_M\in\col(\tilde X_M')$ then 
		$\bar \Pi\in \aff(\partial{J_\Lambda}(M))$.
		Indeed, since $ \tilde X_M'(\tilde X_M')^+$ is the projection on $\col(\tilde X_M')$ we have 
		$$U_M'\bar \Pi = \tilde X_M'(\tilde X_M')^+\tilde \Lambda_M=\tilde \Lambda_M.$$
		Moreover, since $\col((\tilde X_M')^+)=\col(\tilde X_M)$ we deduce that $\bar \Pi\in \col(X'\tilde X_M)$. To prove that $\bar \Pi$ is the unique point in the intersection, let us prove that $\col(X'\tilde{X}_M) \cap \col(U_M)^\perp=\{0\}$. Indeed, if $v\in \col(X'\tilde{X}_M) \cap \col(U_M)^\perp$ 
		then $v=X'\tilde{X}_Mz$ for some $z\in \R^k$ and 
		$U_M'v=0$. Therefore, $\tilde X_M'\tilde{X}_Mz=0$,  consequently  $\tilde X_Mz=0$ and thus $v=\{0\}$. Finally,  if $\Pi \in \aff(\partial{J_\Lambda}(M))\cap \col(X'\tilde X_M)$ then $\Pi-\bar \Pi\in \col(X'\tilde X_M)$ and 
		$U_M'(\Pi-\bar \Pi)=0$ which implies that $\Pi=\bar \Pi$ and establishes ii).\\\\
		According to Corollary \ref{cor:mainmr0},
		pattern recovery by SLOPE in the noiseless case is equivalent to $\bar \Pi \in \partial{J_\Lambda}(M)$ which is equivalent, by i) and ii), to 
		$\col(X'\tilde{X}_M)\cap \partial{J_\Lambda}(M)\not=\emptyset$. 
	\end{proof}
	
	
	\begin{example}
 		\leavevmode
 		\makeatletter
 		\@nobreaktrue
 		\makeatother
		\begin{itemize}
			\item
			We observe on the right picture in Fig. \ref{fig:OWL_path} that the noiseless pattern recovery occurs when $\bar\beta=(5,3)'$ (thus  $M=\model(\bar\beta)=(2,1)'$). To corroborate this fact note that  $\tilde X_M=X$
			thus  $\col(X'\tilde X_{M})=\R^2$ and consequently  $\col(X'\tilde X_{M})$ intersects $\partial{J_\Lambda}(M)$.
			\item
			We observe on the left picture in Fig. \ref{fig:OWL_path} that the noiseless pattern recovery does not occur when $\beta=(5,0)'$ (thus  $M=\model(\beta)=(1,0)'$). To corroborate this fact,  Figure \ref{fig:geom} illustrates that $\col(X'\tilde X_M)=\col((1,0.6)')$ does not intersect $\partial{J_\Lambda}(M)=\{4\}\times [-2,2]$. 
		\end{itemize}
	\end{example}

	\begin{figure}[h!]
		\centering
		\begin{tikzpicture}[scale=0.4]
			\draw[very thin,color=lightgray] (-5.1,-5.1) grid (5.1,5.1);
			\draw[->] (-5.2,0) -- (5.2,0) node[right] {$\beta_1$};
			\draw[->] (0,-5.2) -- (0,5.2) node[above] {$\beta_2$};
			\begin{scope}[ultra thick]
				\draw[color=black] (4,2) -- (2,4) -- (-2,4) -- (-4,2) -- (-4,-2) -- (-2,-4) -- (2,-4) -- (4,-2) -- (4,2);
			\end{scope}
			\draw[color=blue] (-5.2,-3.12) -- (5.2,3.12); 
			\draw[color=red, line width=1mm] (4,-2) -- (4,2);
			\draw[dashed, color=red, line width=0.5mm] (4,2) -- (4,5.2);
			\draw[dashed, color=red, line width=0.5mm] (4,-2) -- (4,-5.2);
			\draw (5.8,1) node[color=red]{$\partial{J_\Lambda}(M)$};
			\draw (5.8,1) node[color=red]{$\partial{J_\Lambda}(M)$};
			\draw (7,3.7) node[color=blue]{$\col(X'\tilde X_M)$};
			\draw (4,2.4)
			node[color=purple,circle,fill,inner sep=0pt,minimum size=3pt]{};
			\draw (4.8,2.3) node[color=purple]{$\bar \Pi$};
		\end{tikzpicture}        
		\caption{This figure illustrates $\bar \Pi$ in purple as the unique intersection point between $\col(X'\tilde X_M)=\col((1,0.6)')$ in blue and  $\aff(\partial{J_\Lambda}(M))$ in red. Since $\bar \Pi \notin \partial{J_\Lambda}(M)=\{4\}\times [-2,2]$ 
			then, in the noiseless case, SLOPE cannot recover $M=\model(\beta)=(1,0)'$. }
		\label{fig:geom}
	\end{figure}
	
\bibliographystyle{elsarticle-num-names} 
\bibliography{biblio_pattern}





\end{document}